\documentclass[10pt]{amsart}

\usepackage{graphicx}              
\usepackage{amsmath}               
\usepackage{amsfonts}              
\usepackage{amsthm}                
\usepackage{breqn} 
\usepackage{xcolor}
\usepackage[pagebackref=false,pdftex,bookmarksopen=true,colorlinks=true, linkcolor=red, citecolor=green]{hyperref}

\usepackage{marginnote}

\newtheorem{thm}{Theorem}[section]
\newtheorem{lem}[thm]{Lemma}
\newtheorem{prop}[thm]{Proposition}

\newcommand{\R}{\mathbb{R}}
\newcommand{\E}{\mathbb{E}}
\newcommand{\J}{\mathbb{J}}      
\newcommand{\K}{\mathbb{K}}      
\newcommand{\h}{\mathbb{H}}            
\newcommand{\s}{\mathbb{S}}      
\newcommand{\T}{\mathbb{T}}

\newcommand{\LL}{\mathcal{L}}

\begin{document}

\title[Decay Rates of Solutions to the Thermoelastic Bresse System]{Decay Rates of the Solutions to the Thermoelastic Bresse System of Types I and III}

\author[Gallego]{F. A. Gallego}
\address{Institute of Mathematics\\
Federal University of Rio de Janeiro, UFRJ \\
P.O. Box 68530, CEP 21945-970, Rio de Janeiro, RJ, Brazil.}
\email{fgallego@ufrj.br, ferangares@gmail.com}

\author[Mu\~{n}oz Rivera]{J. E. Mu\~{n}oz Rivera}
\address{Laborat\'orio de Computuç\~ao Cient\'ifica, LNCC \\ 
Petr\'opolis, RJ, Brazil \\
Federal University of Rio de Janeiro, UFRJ \\
P.O. Box 68530, CEP 21945-970, Rio de Janeiro, RJ, Brazil.}
\email{rivera@lncc.br, rivera@im.ufrj.br}

\date{}


\begin{abstract}
In this paper, we study the energy decay for the thermoelastic Bresse system in the whole line with two different dissipative mechanism, given by heat conduction (Types I and III). We prove that the decay rate of the solutions are very slow. More precisely, we show that the solutions decay with the rate of $(1+t)^{-\frac{1}{8}}$ in the $L^2$-norm, whenever the initial data belongs to $L^1(\R) \cap H^{s}(\R)$ for a suitable $s$. The wave speeds of propagation have influence on the decay rate with respect to the regularity of the initial data. This phenomenon is known as \textit{regularity-loss}. The main tool used to prove our results is the energy method in the Fourier space. 
\end{abstract}

\maketitle

\nocite{*}
\section{Introduction}

In this paper, we consider two Cauchy problems related to the Bresse model with two different dissipative mechanisms, corresponding to the heat conduction coupled to the system. The first of them  is the Bresse system with thermoelasticity of Type I:

\begin{align}\label{ee1}
\left\lbrace\begin{array}{r l}
\rho_1\varphi _{tt}-k \left ( \varphi_x - \psi -l\omega\right )_x -k_0l\left ( \omega_x - l\varphi \right )+l\gamma \theta_1 = 0 &  \text{in $\R \times (0,\infty)$},  \\
\rho_2\psi_{tt}-b\psi_{xx}-k\left ( \varphi_x - \psi -l\omega\right )+\gamma \theta_{2x} =0   &\text{in $\R \times (0,\infty)$},  \\
\rho_1\omega_{tt}-k_0\left (\omega_x - l \varphi \right )_x- kl\left ( \varphi_x - \psi -l\omega\right )+\gamma \theta_{1x} =0 & \text{in $\R \times (0,\infty)$},\\
\theta_{1t} -k_1\theta_{1xx}+m_1\left (\omega_x - l \varphi \right )_t=0 &\text{in $\R \times (0,\infty)$}, \\
\theta_{2t} -k_2\theta_{2xx}+m_2\psi_{xt}=0  & \text{in $\R \times (0,\infty)$}, 
\end{array}\right.
\end{align}

\noindent with the initial data
\begin{equation*}
(\varphi, \varphi_t, \psi, \psi_t,\omega, \omega_t,\theta_1,\theta_2)(x,0)=(\varphi_0, \varphi_1, \psi_0, \psi_1,\omega_0, \omega_1,\theta_{10},\theta_{20})(x).
\end{equation*}
The second one, is the Bresse system with thermoelasticity of Type III:

\begin{align}\label{ee2}
\left\lbrace \begin{array}{r l}
\rho_1\varphi _{tt}-k \left ( \varphi_x - \psi -l\omega\right )_x -k_0l\left ( \omega_x - l\varphi \right )+l\gamma \theta_{1t} = 0 & \text{in $\R \times (0,\infty)$},  \\
\rho_2\psi_{tt}-b\psi_{xx}-k\left ( \varphi_x - \psi -l\omega\right )+\gamma \theta_{2xt} =0 & \text{in $\R \times (0,\infty)$},  \\
\rho_1\omega_{tt}-k_0\left (\omega_x - l \varphi \right )_x- kl\left ( \varphi_x - \psi -l\omega\right )+\gamma \theta_{1xt} =0 & \text{in $\R \times (0,\infty)$},\\
\theta_{1tt} -k_1\theta_{1xx}-\alpha_1\theta_{1xxt}+m_1\left (\omega_x - l \varphi \right )_t=0  &  \text{in $\R \times (0,\infty)$}, \\
\theta_{2tt} -k_2\theta_{2xx}-\alpha_2\theta_{2xxt}+m_2\psi_{xt}=0   &  \text{in $\R \times (0,\infty)$}, 
\end{array}\right.
\end{align}

\noindent with the initial data
\begin{equation*}
(\varphi, \varphi_t, \psi, \psi_t,\omega, \omega_t,\theta_1,\theta_2, \theta_{1t},\theta_{2t})(x,0)=(\varphi_0, \varphi_1, \psi_0, \psi_1,\omega_0, \omega_1,\theta_{10},\theta_{20},\theta_{11},\theta_{21})(x),
\end{equation*}
where $\alpha_1$, $\alpha_2$, $\rho_1,\rho_2, \gamma, b, k, k_0,k_1,k_2,l$ $m_1$ and $m_2$ are positive constants. 
\vglue 0.2cm

The terms $k_0(\omega_x-l\varphi)$, $k(\varphi-\psi-l\omega)$ and $b\psi_x$  denote the axial force, the shear force and the bending moment, where $\omega$, $\varphi$  and $\psi$ are the longitudinal, vertical and shear angle displacements, respectively. Furthermore, $\rho_1 = \rho A$, $\rho_2 = \rho I$, $k_0 = EA$, $k = k'GA$, $b = EI$ and $l = R^{-1}$, where $\rho$ denotes the density, $E$ is the elastic modulus, $G$ is the shear modulus, $k'$ is the shear factor, $A$ is the cross-sectional area, $I$ is the second moment of area of the cross-section and $R$ is the radius of curvature of the beam. Here, we assume that
all the above coefficients are positive constants.  In what concerns of the Thermoelastic of type III, we refer the work of Green and Naghdi  \cite{green1991re, green1992undamped}. They re-examined the classical model of thermoelasticity and introduced the so-called model of thermoelasticity of type III, which the constitutive assumption on the heat flux vector is different from Fourier's law. They  developed a model of thermoelasticity that includes temperature gradient and thermal displacement gradient among the constitutive variables and proposed a heat conduction law as
\begin{equation}\label{t1}
q(x,t)=-(\kappa  \theta_x (x,t)+\kappa^* v_x(x,t)),
\end{equation}
where $v_t = \theta$ and $v$ is the thermal displacement gradient, $\kappa$ and $\kappa^*$ are constants.
Combining (\ref{t1}) with the energy balance law
\begin{equation}\label{t2}
\rho\theta_{t} + \varrho  \, div\, q = 0, 
\end{equation}
lead to the equation
\begin{equation*}
\rho \theta_{tt} - \varrho \kappa  \theta_{xx} -\varrho \kappa^*  \theta_{xx} = 0,
\end{equation*}
which permits propagation of thermal waves at finite speed. The common feature of these theories, is that all of them lead to hyperbolic differential equations and model heat flow as thermal waves traveling at finite speed. More information about mathematical modeling can be found in \cite{chandrasekharaiah1998hyperbolic, green1992undamped, lagnese1993}.
\vglue 0.2cm

The main purpose of this paper is to investigate the asymptotic behavior of the solutions to the Cauchy problems  \eqref{ee1} and \eqref{ee2} posed on $\R$. To the best of our knowledge, the stability of the Bresse model does not have any phisycal explanation when it is considered in the real line. Be that  as it may, from  mathematical point of view, a considerable number of stability issues concerning the Bresse model in a whole space, have received considerable attention in the last years \cite{ghoul2016,  said2015asymptotic, 2015bressetypeiii, said2014bresse, soufyane2014effect}. This has been due to \textit{the regularity-loss phenomenon} that usually appears in the pure Cauchy problems (for instance, see \cite{djou2014, duan2001, haramotokawashima2008, hosono2006, idekawashima2008, ueda2011} and references therein). Roughly speaking, the decay rate of the solution is of the \textit{regularity-loss type}, when it is obtained only by assuming some additional order regularity on the initial date. Thus,  based on this refinement  of the initial data, we investigate the relationship between damping terms, the wave speeds of propagation and their influence on the decay rate of the vector solutions $V_1$ and $V_2$ (see \eqref{vectorsolution1}-\eqref{vectorsolution2} below) of the systems \eqref{ee1} and \eqref{ee2}, respectively. 
\vglue 0.3cm

Thus, our main result reads as follows:
\begin{thm}\label{teo2}
Let $s$ be a nonnegative integer, suppose that $V^0_j \in H^s(\R)\cap L^1(\R)$ for $j=1,2$. Then, the vector solutions $V_j$ of thermoelastic Bresse problems  $(\ref{ee1})$ and $(\ref{ee2})$, respectively, satisfy the following decay estimates,
\begin{enumerate}
\item If $\frac{\rho_1}{\rho_2} =\frac{k}{b}$ and $k=k_0$, then
\begin{align}\label{e32}
\|\partial^k_xV_j(t)\|_2\leq C_1(1+t)^{-\frac{1}{8}-\frac{k}{4}}\|V_j^0\|_1 + C_2(1+t)^{-\frac{l}{4}}\|\partial_x^{k+l}V_j^0\|_2,  \qquad j=1,2, \quad t\geq 0.
\end{align}
\item If $\frac{\rho_1}{\rho_2} \neq\frac{k}{b}$ or $k\neq k_0$, then
\begin{align}\label{e32'}
\|\partial^k_xV_j(t)\|_2\leq C_1(1+t)^{-\frac{1}{8}-\frac{k}{4}}\|V_j^0\|_1 + C_2(1+t)^{-\frac{l}{6}}\|\partial_x^{k+l}V_j^0\|_2, \qquad j=1,2, , \quad t\geq 0.
\end{align}
\end{enumerate}
where $k+l \leq s$, $C_1,C_2$ are two positive constants.
\end{thm}

Our proof is based on some estimates for the Fourier image of the solution  as well as a suitable linear combination of series of energy estimates. The key idea is to construct functionals to capture the dissipation of all the components of the vector solution. These functional allows to build an appropriate Lyapunov functionals equivalent to the energy, which gives the dissipation of all the components in the vector $\hat{V}^0_1(\xi,t)$ and $\hat{V}^0_2(\xi,t)$ (See \eqref{eq45} below). Finally, we rely on the Plancherel theorem and some asymptotic inequalities to show the desired decay estimates. 

The decay rate $(1 + t)^{-\frac{1}{8}}$  can be obtained only under the regularity $V_0 \in  H^s(\R)$. This regularity loss comes to analyze the Fourier image of the solution. Indeed, for $\hat{V}(\xi, t)$, we have (see \eqref{eq27'}, \eqref{eq27} and \eqref{eq45} below) that 
\begin{equation}
\left | \hat{V}(\xi, t)\right|^2 \leq C e^{-\beta s(\xi)t}\left | \hat{V}(\xi, 0)\right|^2,
\end{equation}
where
\begin{equation*}
s(\xi) = \left\lbrace\begin{tabular}{l l}
$C_1\dfrac{\xi^4}{\left(1+\xi^8\right)}$, & if $\frac{\rho_1}{\rho_2} =\frac{k}{b}$ and $k=k_0$, \\
$C_2\dfrac{\xi^4}{\left(1+\xi^2\right)\left(1+\xi^8\right)}$, & if $\frac{\rho_1}{\rho_2} \neq\frac{k}{b}$ or $k\neq k_0.$
\end{tabular}\right.
\end{equation*}
\vglue 0.2 cm
As we will see, the decay estimate \eqref{e32}-\eqref{e32'} depends in a critical way on the properties of the function $s(\xi)$. Obviously,
the function $s(\xi)$ behaves like $\xi^4$ in the low frequency region $(|\xi| \leq 1)$ and like $\xi^{-4}$ near infinity whenever $\frac{\rho_1}{\rho_2} =\frac{k}{b}$ and $k=k_0$. Otherwise, if the wave speeds of propagation are different the function $s(\xi)$ behaves also like $\xi^4$ in the low frequency region and like $\xi^{-6}$ near infinity, which means that the dissipation in the high frequency region is very weak and produces the regularity loss phenomenom. It has been known recently that this regularity loss leads to some difficulties in the nonlinear problems, see \cite{haramotokawashima2008, idekawashima2008} for more details.
\vglue 0.2cm


There are many works on the global existence and asymptotic stability of solutions to the initial boundary value problem for the Bresse system with dissipation. In this direction, we refer the IBVP associated to (\ref{ee1}) considered by Liu and Rao in \cite{liu2009energy}. They proved that the exponential decay  exists only when the velocities of the wave propagation are the same. If the wave speeds are different, they showed that the energy  decays polynomially to zero with the rate $t^{-\frac{1}{2}}$ and $t^{-\frac{1}{4}}$, provided that the boundary conditions are Dirichlet-Neumann-Neumann 
\[
\omega_x(x,t)=\varphi(x,t)=\psi_x(x,t)=\theta_1(x,t)=\theta_2(x,t)=0, \quad \text{for $x=0,l$},
\]
and Dirichlet-Dirichlet-Dirichlet type, 
\[
\omega(x,t)=\varphi(x,t)=\psi(x,t)=\theta_1(x,t)=\theta_2(x,t)=0, \quad \text{for $x=0,l$.}
\]
An improvement of the above results was made by Fatori and Mu\~{n}oz Rivera in \cite{fatori2010rates}. They showed that, in general, the Thermoelastic Bresse system of Type I is not exponentially stable, but there exists polynomial stability with rates that depend on the wave propagations and the regularity of the initial data. 

\vglue 0.2cm

As far as know, there exist few results related to the stability of the pure Cauchy problem to the Bresse model. The decay rate of the solution of the IVP for Bresse system in the whole line has been first studied by Said-Houari and Soufyane in \cite{soufyane2014effect}. They considered the system
\begin{equation}\label{eee1}
\left\lbrace\begin{tabular}{r l}
$\varphi _{tt}- \left ( \varphi_x - \psi -l\omega\right )_x -k_0^2l\left ( \omega_x - l\varphi \right ) = 0$  & in $\R \times (0,\infty)$,  \\
$\psi_{tt}-a^2\psi_{xx}-k\left ( \varphi_x - \psi -l\omega\right )\gamma_1 \psi_t =0$  & in $\R \times (0,\infty)$, \\
$\omega_{tt}-k_0^2\left (\omega_x - l \varphi \right )_x- l\left ( \varphi_x - \psi -l\omega\right )+\gamma_2 \omega_t =0$  & in $\R \times (0,\infty),$
\end{tabular}\right.
\end{equation}
and investigated the relationship between the frictional damping terms, the wave speeds of propagation and their influence on the decay rate of the solution. In addition, they showed that the $L^2$-norm of the solution decays with the rate $(1 + t)^{-1/4}$. Later on, the same authors in \cite{said2014bresse}, proved that the vector solution $V$ of the Bresse system damped by heat conduction:
\begin{equation}\label{eee2}
\left\lbrace\begin{tabular}{r l}
$\varphi _{tt}-k \left ( \varphi_x - \psi -l\omega\right )_x -k_0^2l\left ( \omega_x - l\varphi \right ) = 0$  & in $\R \times (0,\infty)$,  \\
$\psi_{tt}-a^2\psi_{xx}-k\left ( \varphi_x - \psi -l\omega\right )+m \theta_x =0$  & in $\R \times (0,\infty)$, \\
$\omega_{tt}-k_0^2\left (\omega_x - l \varphi \right )_x- l\left ( \varphi_x - \psi -l\omega\right )+\gamma \omega_t =0$  & in $\R \times (0,\infty)$, \\
$\theta_{t} -k_1\theta_{xx}+m\psi_{xt}=0$  & in $\R \times (0,\infty)$, 
\end{tabular}\right.
\end{equation}
decays with the rate,
\begin{equation}\label{eee3}
\|\partial_x^kV(t)\|_{L^2} \leq C(1+t)^{-\frac{1}{12}-\frac{k}{6}}\|V_0\|_{L^1}+C(1+t)^{-\frac{l}{2}}\|\partial_x^{k+l}V_0\|_{L^2},
\end{equation}
for $a=1$, and 
\begin{equation}\label{eee4}
\|\partial_x^kV(t)\|_{L^2} \leq C(1+t)^{-\frac{1}{12}-\frac{k}{6}}\|V_0\|_{L^1}+C(1+t)^{-\frac{l}{4}}\|\partial_x^{k+l}V_0\|_{L^2},
\end{equation}
for $a\neq 1$, $k=1,2,...,s-l$.
More recently, Said-Houari and Hamadouche  \cite{said2015asymptotic} studied the decay properties of the Bresse-Cattaneo system:
\begin{equation}\label{eee5}
\left\lbrace\begin{tabular}{r l}
$\varphi _{tt}-\left ( \varphi_x - \psi -l\omega\right )_x -k_0^2l\left ( \omega_x - l\varphi \right ) = 0$  & in $\R \times (0,\infty)$,  \\
$\psi_{tt}-a^2\psi_{xx}-\left ( \varphi_x - \psi -l\omega\right )+m \theta_x =0$  & in $\R \times (0,\infty)$, \\
$\omega_{tt}-k_0^2\left (\omega_x - l \varphi \right )_x- l\left ( \varphi_x - \psi -l\omega\right )+\gamma \omega_t =0$  & in $\R \times (0,\infty)$, \\
$\theta_{t} +q_x+m\psi_{xt}=0$  & in $\R \times (0,\infty)$, \\
$\tau_q q_t  +\beta q+\theta_{x}=0$  & in $\R \times (0,\infty)$, 
\end{tabular}\right.
\end{equation}
obtaining the same decay rate as the one of the solution for the Bresse-Fourier model (\ref{eee2}).  This fact has been also seen in the paper \cite{said2013damping}, where the authors investigated the Timoshenko-Cattaneo and Timoshenko-Fourier models and showed the same behavior for the solutions of both systems. Finally, concerning to the Termoelasticity type III (in one-dimensional space), Said-Houari and Hamadouche in \cite{2015bressetypeiii} have been recently analyzed the system:
\begin{equation*}
\left\lbrace\begin{tabular}{r l}
$\varphi _{tt}-\left ( \varphi_x - \psi -l\omega\right )_x -k_0^2l\left ( \omega_x - l\varphi \right ) = 0$  & in $\R \times (0,\infty)$,  \\
$\psi_{tt}-a^2\psi_{xx}-\left ( \varphi_x - \psi -l\omega\right )+m \theta_{tx} =0$  & in $\R \times (0,\infty)$, \\
$\omega_{tt}-k_0^2\left (\omega_x - l \varphi \right )_x- l\left ( \varphi_x - \psi -l\omega\right )+\gamma \omega_t =0$  & in $\R \times (0,\infty)$, \\
$\theta_{tt} -k_1\theta_{xx}+\beta\psi_{tx}-k_2\theta_{txx}=0$  & in $\R \times (0,\infty).$ 
\end{tabular}\right.
\end{equation*}
They proved that the solution decay with the rate:
\begin{equation*}
\|\partial_x^kV(t)\|_{L^2} \leq C(1+t)^{-\frac{1}{12}-\frac{k}{6}}\|V_0\|_{L^1}+C(1+t)^{-\frac{l}{2}}\|\partial_x^{k+l}V_0\|_{L^2},
\end{equation*}
for $a=1$, and 
\begin{equation*}
\|\partial_x^kV(t)\|_{L^2} \leq C(1+t)^{-\frac{1}{12}-\frac{k}{6}}\|V_0\|_{L^1}+C(1+t)^{-\frac{l}{8}}\|\partial_x^{k+l}V_0\|_{L^2},
\end{equation*}
for $a\neq 1$, $k=1,2,...,s-l$. 
\vglue 0.4cm

This paper is organized as follows:
\vglue 0.2cm

- In Section 2, we analyze the ODE system generated by the Fourier transform applies to the Cauchy problem, obtaining  appropriate decay estimates for the Fourier image of the solution. 
\vglue 0.2cm

- Section 3 is dedicated to proof our main result.

\section{Energy method in the Fourier space.}
 In this section, we establish  decay rates for the Fourier image of the solutions of thermoelastic Bresse systems.  To obtain the estimates of the Fourier image is actually the hardest and technical part. These estimates will play to a crucial role in proving the Theorems \ref{teo1} and \ref{teo1'}, below.

\subsection{Thermoelastic Bresse system of Type I}  
Taking Fourier Transform in (\ref{ee1}), we obtain the following ODE system:

\begin{align}
\rho_1\hat{\varphi} _{tt}-ik\xi \left ( i\xi \hat{\varphi} - \hat{\psi} -l \hat{\omega }  \right ) -k_0l\left ( i\xi \hat{\omega } - l \hat{\varphi} \right ) +l\gamma \hat{\theta}_1 &= 0 \quad \text{in $\R \times (0,\infty)$} \label{e1}\\
\rho_2\hat{\psi}_{tt}+b\xi ^{2}\hat{\psi }- k\left ( i\xi \hat{\varphi} - \hat{\psi} -l \hat{\omega }  \right )+i\gamma \xi  \hat {\theta}_2&=0 \quad \text{in $\R \times (0,\infty)$}\label{e2}\\
\rho_1\hat{\omega}_{tt}-ik_0\xi \left ( i\xi \hat{\omega } - l \hat{\varphi} \right )- kl\left ( i\xi \hat{\varphi} - \hat{\psi} -l \hat{\omega }  \right )+ i\gamma \xi \hat{\theta}_1&=0 \quad \text{in $\R \times (0,\infty)$}\label{e3}\\
\hat{\theta}_{1t} +k_1\xi ^{2}\hat{\theta}_1+m_1 \left ( i\xi \hat{\omega } - l \hat{\varphi} \right )_t&=0 \quad \text{in $\R \times (0,\infty)$}\label{e4} \\
\hat{\theta}_{2t} +k_2\xi ^{2}\hat{\theta}_2+im_2 \xi \hat{\psi}_{t} &=0 \quad \text{in $\R \times (0,\infty)$}\label{e5}
\end{align}
The energy functional associated to the above system is defined as:
\begin{equation}\label{energytypeII}
\hat{E}\left ( \xi,t  \right )=\rho_1|\hat{\varphi}_{t} |^{2}+\rho_2|\hat{\psi}_{t}|^{2}+\rho_1 |\hat{\omega}_{t}|^{2}+\frac{\gamma}{m_1}|\hat{\theta}_1 |^{2}+\frac{\gamma}{m_2}|\hat{\theta}_2 |^{2}+b| \xi|^{2}|\hat{\psi}|^{2}+k|i\xi \hat{\varphi} - \hat{\psi} -l \hat{\omega }|^{2} + k_0|i\xi \hat{\omega } - l \hat{\varphi}|^{2}.
\end{equation}

\begin{lem}\label{lem1}
Consider the energy functional $\hat{E}$ associated to the system \eqref{e1}-\eqref{e5}. Then, 
\begin{gather}\label{e7}
\frac{d}{dt}\hat{E}(\xi,t)=-2\gamma\xi^2\left(\frac{k_1}{m_1}|\hat{\theta}_1|^2+\frac{k_2}{m_2}|\hat{\theta}_2|^2\right).
\end{gather}
\end{lem}
\begin{proof}
Multiplying $(\ref{e1})$ by $\overline{\hat{\varphi}}_t$,  $(\ref{e2})$  by  $\overline{\hat{\psi}}_t$,  $(\ref{e3})$  by $\overline{\hat{\omega}}_t$,  $(\ref{e4})$  by $\frac{\gamma}{m_1}\overline{\hat{\theta}_1}$, and $(\ref{e5})$ by $\frac{\gamma}{m_2}\overline{\hat{\theta}_2}$, adding and  taking real part, $(\ref{e7})$ follows.

\end{proof}
We show that the decay rate of the solution will depend on the wave speeds of propagation. More precisely, we analyze two cases: First, we suppose that 
\[
\frac{\rho_1}{\rho_2}=\frac{k}{b} \quad \text{and} \quad k=k_0.
\]
Otherwise, we consider the case when the wave speeds of propagation are different ($\frac{\rho_1}{\rho_2}\neq\frac{k}{b}$ or $k\neq k_0$). The proof of our main results in this section (Theorems $\ref{teo1}$ and $\ref{teo1'}$ below) are based on the following lemmas:

\begin{lem}\label{lem2}
The functional
\begin{equation*}
J_1(\xi,t)=Re(i\rho_2\xi \hat{\psi}_t\overline{\hat{\theta}_2}),
\end{equation*}
satisfies
\begin{align}\label{e9}
\frac{d}{dt}J_1(\xi,t) + \frac{m_2\rho_2}{2}\xi^2|\hat{\psi}_t|^2 \leq b|\xi|^3|\hat{\psi}||\hat{\theta}_2| + k|\xi||\hat{\theta}_2| | i\xi \hat{\varphi} - \hat{\psi} -l \hat{\omega }|+  C_1(1+\xi^2)\xi^2|\hat{\theta}_2|^2,\end{align}
where $C_1$ is a positive constant.
\end{lem}
\begin{proof}
Multiplying $(\ref{e5})$ by $-i\rho_2\xi \overline{\hat{\psi}_t}$ and taking real part, we obtain
\begin{align*}
\frac{d}{dt}Re\left(-i\rho_2\xi \overline{\hat{\psi}_t}\hat{\theta}_{2}\right)+Re\left(i\rho_2\xi \overline{\hat{\psi}_{tt}}\hat{\theta}_{2}\right) -Re\left(ik_2\rho_2\xi^3 \overline{\hat{\psi}_t}\hat{\theta}_2\right)+m_2\rho_2 \xi^2 |\hat{\psi}_{t}|^2 =0.
\end{align*}
By $(\ref{e2})$, we have
\begin{align*}
\frac{d Re}{dt}\left(-i\rho_2\xi \overline{\hat{\psi}_t}\hat{\theta}_{2}\right)&+m_2\rho_2 \xi^2 |\hat{\psi}_{t}|^2 \\
&\leq k_2\rho_2|\xi|^3 |\hat{\psi}_t||\hat{\theta}_2|+b|\xi|^3 |\hat{\psi}| |\hat{\theta}_{2}| + k|\xi||\hat{\theta}_{2}|| i\xi \hat{\varphi} - \hat{\psi} -l \hat{\omega } |+\gamma \xi^2 |\hat {\theta}_2|^2
\end{align*}
Applying Young inequality, we obtain $(\ref{e9})$.
\end{proof}


\begin{lem}\label{lem3'}
The functional
\begin{equation*}
T_1(\xi,t)=Re\left(-\rho_1\hat{\varphi}_t \overline{\left(i\xi \hat{\omega } - l \hat{\varphi}\right)} - \frac{\rho_1}{m_1}\hat{\varphi}_t\overline{\hat{\theta}_1}\right),
\end{equation*}
satisfies
\begin{align}\label{e10}
\frac{d}{dt}T_1(\xi,t) +\frac{k_0l}{2}|i\xi \hat{\omega } - l \hat{\varphi}|^2 \leq &\frac{\rho_1k_1}{m_1}|\xi|^2|\hat{\varphi}_t||\hat{\theta}_1|-Re(ik\xi(i\xi\hat{\varphi}-\hat{\psi}-l\hat{\omega})\overline{\left(i\xi \hat{\omega } - l \hat{\varphi}\right)}) \notag \\
&+\frac{k}{m_1}|\xi||\hat{\theta}_1||i\xi\hat{\varphi}-\hat{\psi}-l\hat{\omega}|+C_2 |\hat{\theta}_1|^2,
\end{align}
where $C_2$ is a positive constant.
\end{lem}

\begin{proof}
Multiplying (\ref{e1}) by $-\overline{\left(i\xi \hat{\omega } - l \hat{\varphi}\right)}$ and taking real part, we have

\begin{align*}
\frac{d}{dt} Re(-\rho_1\hat{\varphi} _{t}\overline{\left(i\xi \hat{\omega } - l \hat{\varphi}\right)})+Re(\rho_1\hat{\varphi} _{t}\overline{\left(i\xi \hat{\omega } - l \hat{\varphi}\right)}_t)&+Re(ik\xi \left ( i\xi \hat{\varphi} - \hat{\psi} -l \hat{\omega }  \right)\overline{\left(i\xi \hat{\omega } - l \hat{\varphi}\right)}) \notag\\
&+k_0l|i\xi \hat{\omega } - l \hat{\varphi}|^2 -Re(l\gamma \hat{\theta}_1\overline{\left(i\xi \hat{\omega } - l \hat{\varphi}\right)})= 0.
\end{align*}
(\ref{e4}) implies that
\begin{align}\label{e34}
\frac{d Re}{dt}(-\rho_1\hat{\varphi} _{t}\overline{\left(i\xi \hat{\omega } - l \hat{\varphi}\right)})&-\frac{\rho_1}{m_1}Re\left(\hat{\varphi}_t\overline{\hat{\theta}}_{1t}\right) -\frac{\rho_1k_1}{m_1}Re\left(\xi^2\hat{\varphi}_t\overline{\hat{\theta}}_{1}\right) \notag \\
&+Re(ik\xi \left ( i\xi \hat{\varphi} - \hat{\psi} -l \hat{\omega }  \right)\overline{\left(i\xi \hat{\omega } - l \hat{\varphi}\right)}) +k_0l|i\xi \hat{\omega } - l \hat{\varphi}|^2 -Re(l\gamma \hat{\theta}_1\overline{\left(i\xi \hat{\omega } - l \hat{\varphi}\right)})= 0.
\end{align}
On the other hand, multiplying (\ref{e1}) by $-\dfrac{\overline{\hat{\theta}}_1}{m_1}$ and taking real part, it follows that 
\begin{align}\label{e35}
\frac{d Re}{dt}\left(-\frac{\rho_1}{m_1}\hat{\varphi} _{t}\overline{\hat{\theta}}_1\right)+\frac{\rho_1}{m_1}Re\left(\hat{\varphi} _{t}\overline{\hat{\theta}}_{1t}\right)+Re\left( \frac{ik}{m_1}\xi \overline{\hat{\theta}}_1 \left ( i\xi \hat{\varphi} - \hat{\psi} -l \hat{\omega }  \right )\right) +Re\left(\frac{k_0l}{m_1}\overline{\hat{\theta}}_1\left ( i\xi \hat{\omega } - l \hat{\varphi} \right ) \right)&  \notag\\
-\frac{l\gamma}{m_1} |\hat{\theta}_1|^2 &= 0.
\end{align}
Adding (\ref{e34}) and (\ref{e35}), 
\begin{align*}
\frac{d}{dt}T_1(\xi,t)+k_0l|i\xi \hat{\omega } - l \hat{\varphi}|^2 \leq &\frac{\rho_1k_1}{m_1}|\xi|^2|\hat{\varphi}_t||\hat{\theta}_{1}|-Re\left( ik\xi \left ( i\xi \hat{\varphi} - \hat{\psi} -l \hat{\omega }  \right)\overline{\left(i\xi \hat{\omega } - l \hat{\varphi}\right)}\right)+l\gamma |\hat{\theta}_1||i\xi \hat{\omega } - l \hat{\varphi}| \notag \\
&+\frac{k}{m_1}|\xi| |\hat{\theta}_1||i\xi \hat{\varphi} - \hat{\psi} -l \hat{\omega }| +\frac{k_0l}{m_1}|\hat{\theta}_1|| i\xi \hat{\omega } - l \hat{\varphi}|+\frac{l\gamma}{m_1} |\hat{\theta}_1|^2, 
\end{align*}
applying Young inequality, (\ref{e10}) follows.
\end{proof}

\begin{lem}\label{lem3''}
The functional
\begin{equation*}
T_2(\xi,t)=Re\left(i\rho_1\xi\hat{\omega}_t \overline{\left(i\xi \hat{\omega } - l \hat{\varphi}\right)} + i\frac{\rho_1}{m_1}\xi\hat{\omega}_t\overline{\hat{\theta}_1}\right),
\end{equation*}
satisfies
\begin{align}\label{e36}
\frac{d}{dt}T_2(\xi,t) +\frac{k_0}{2}|\xi|^2|i\xi \hat{\omega } - l \hat{\varphi}|^2 \leq &\frac{\rho_1k_1}{m_1}|\xi|^3|\hat{\omega}_t||\hat{\theta}_1|+Re(ikl\xi(i\xi\hat{\varphi}-\hat{\psi}-l\hat{\omega})\overline{\left(i\xi \hat{\omega } - l \hat{\varphi}\right)}) \notag \\
&+\frac{kl}{m_1}|\xi||\hat{\theta}_1||i\xi\hat{\varphi}-\hat{\psi}-l\hat{\omega}|+C_3 |\xi|^2|\hat{\theta}_1|^2,
\end{align}
where $C_3$ is a positive constant.
\end{lem}
\begin{proof}

Multiplying (\ref{e3}) by $i\xi\overline{\left(i\xi \hat{\omega } - l \hat{\varphi}\right)}$ and taking real part, we obtain
\begin{align}\label{e37}
\frac{d}{dt} Re\left(i\rho_1\xi\hat{\omega}_{t}\overline{\left(i\xi \hat{\omega } - l \hat{\varphi}\right)}\right)&-Re\left(i\rho_1\xi\hat{\omega}_{t}\overline{\left(i\xi \hat{\omega } - l \hat{\varphi}\right)}_t\right)+k_0|\xi|^2 | i\xi \hat{\omega } - l \hat{\varphi}|^2 \notag \\
&- Re\left(ikl\xi\left ( i\xi \hat{\varphi} - \hat{\psi} -l \hat{\omega }  \right )\overline{\left(i\xi \hat{\omega } - l \hat{\varphi}\right)}\right)- Re\left( \gamma \xi^2 \hat{\theta}_1\overline{\left(i\xi \hat{\omega } - l \hat{\varphi}\right)}\right)=0,
\end{align}
using (\ref{e4}), it follows that 
\begin{align}\label{e38}
\frac{d}{dt} Re\left(i\rho_1\xi\hat{\omega}_{t}\overline{\left(i\xi \hat{\omega } - l \hat{\varphi}\right)}\right)&+\frac{\rho_1}{m_1}Re\left(i\xi\hat{\omega}_t\overline{\hat{\theta}}_{1t}\right) +\frac{\rho_1k_1}{m_1}Re\left(i\xi^3\hat{\omega}_t\overline{\hat{\theta}}_{1}\right)+k_0|\xi|^2 | i\xi \hat{\omega } - l \hat{\varphi}|^2 \notag \\
&- Re\left(ikl\xi\left ( i\xi \hat{\varphi} - \hat{\psi} -l \hat{\omega }  \right )\overline{\left(i\xi \hat{\omega } - l \hat{\varphi}\right)}\right)- Re\left( \gamma \xi^2 \hat{\theta}_1\overline{\left(i\xi \hat{\omega } - l \hat{\varphi}\right)}\right)=0.
\end{align}
On the other hand, multiplying (\ref{e3}) by $\dfrac{i\xi}{m_1}\overline{\hat{\theta}}_1$ and taking real part, 
\begin{align}\label{e39}
\frac{d}{dt} Re\left(\frac{i\rho_1\xi}{m_1}\hat{\omega}_{t}\overline{\hat{\theta}}_1\right)&-\frac{\rho_1}{m_1}Re\left(i\xi\hat{\omega}_{t}\overline{\hat{\theta}}_{1t}\right)+Re\left( \frac{k_0\xi^2}{m_1} \left( i\xi \hat{\omega } - l \hat{\varphi}\right)\overline{\hat{\theta}}_1\right) \notag \\
&- Re\left(\frac{ikl\xi}{m_1}\left ( i\xi \hat{\varphi} - \hat{\psi} -l \hat{\omega }  \right )\overline{\hat{\theta}}_1\right)- \frac{\gamma }{m_1}|\xi|^2|\hat{\theta}_1|^2=0. 
\end{align}
Adding (\ref{e38}) and (\ref{e39}), applying Young inequality, (\ref{e36}) follows.
\end{proof}

\begin{lem}\label{lem3}
Consider the functional 
\begin{equation*}
J_2(\xi,t):=lT_1(\xi,t)+T_2(\xi,t).
\end{equation*}
Then, there exist $\delta > 0$ such that 
\begin{align*}
\frac{d}{dt}J_2(\xi,t)+k_0\delta| i\xi \hat{\omega } - l \hat{\varphi}|^2 \leq &\frac{\rho_1lk_1}{m_1} |\xi|^2|\hat{\varphi}_t||\hat{\theta}_{1}|+\frac{\rho_1k_1}{m_1} |\xi|^3|\hat{\omega}_t||\hat{\theta}_{1}| + \frac{2kl}{m_1}|\xi|| i\xi \hat{\varphi} - \hat{\psi} -l \hat{\omega }||\hat{\theta}_1|+C_4(1+\xi^2)|\hat{\theta}_1|^2,
\end{align*}
where $C_4$ is a positive constant.
\end{lem}


\begin{proof}
Lemmas (\ref{lem3'}) and (\ref{lem3''}) imply that
\begin{align*}
\frac{d}{dt}J_2(\xi,t) +\frac{k_0}{2}(l^2+\xi^2)|i\xi \hat{\omega } - l \hat{\varphi}|^2 \leq &\frac{\rho_1lk_1}{m_1}|\xi|^2|\hat{\varphi}_t||\hat{\theta}_1|+\frac{2kl}{m_1}|\xi||\hat{\theta}_1||i\xi\hat{\varphi}-\hat{\psi}-l\hat{\omega}|\notag\\
&+\frac{\rho_1k_1}{m_1}|\xi|^3|\hat{\omega}_t||\hat{\theta}_1|+C_4(1+ \xi^2)|\hat{\theta}_1|^2.
\end{align*}
Note that there exist $\delta >0$ such that $2\delta \leq \frac{l^2+\xi^2}{1+\xi^2}$. Thus,
\begin{align}\label{e43}
\frac{d}{dt}J_2(\xi,t) +k_0\delta(1+\xi^2)|i\xi \hat{\omega } - l \hat{\varphi}|^2 \leq &\frac{\rho_1lk_1}{m_1}|\xi|^2|\hat{\varphi}_t||\hat{\theta}_1|+\frac{2kl}{m_1}|\xi||\hat{\theta}_1||i\xi\hat{\varphi}-\hat{\psi}-l\hat{\omega}|\notag\\
&+\frac{\rho_1k_1}{m_1}|\xi|^3|\hat{\omega}_t||\hat{\theta}_1|+C_4(1+ \xi^2)|\hat{\theta}_1|^2.
\end{align}
\end{proof}

\begin{lem}\label{lem4}
Consider the functional 
\begin{equation*}
J_3(\xi,t)=Re\left(-\rho_2 \hat{\psi}_t\overline{\left ( i\xi \hat{\varphi} - \hat{\psi} -l \hat{\omega }  \right )}-i\frac{\rho_1b}{k}\xi\hat{\psi}\overline{\hat{\varphi}_t}\right).
\end{equation*}
If $\frac{\rho_1}{\rho_2} =\frac{k}{b}$ and $k=k_0$, then
\begin{align}\label{e12}
\frac{d}{dt}J_3(\xi,t) +\frac{k}{2}|i\xi \hat{\varphi} -\hat{\psi} -l \hat{\omega }|^2 \leq \rho_2 |\hat{\psi}_t|^2 +\rho_2 l Re\left( \hat{\psi}_t\overline{\hat{\omega}_t}\right) -blRe \left( i\xi\hat{\psi}\overline{\left ( i\xi \hat{\omega } - l \hat{\varphi} \right )}\right) + \frac{bl\gamma}{k} |\xi| |\hat{\psi}||\hat{\theta}_1|+C_5|\xi|^2|\hat{\theta}_2|^2.
\end{align}
Moreover, if $\frac{\rho_1}{\rho_2} \neq \frac{k}{b}$ or $k\neq k_0$, then
\begin{align}\label{e12'}
\frac{d}{dt}J_3(\xi,t) +\frac{k}{2}|i\xi \hat{\varphi} -\hat{\psi} -l \hat{\omega }|^2 \leq &\rho_2 |\hat{\psi}_t|^2 +\rho_2 l Re\left( \hat{\psi}_t\overline{\hat{\omega}_t}\right) +\left(\rho_2-\frac{b\rho_1}{k}\right) Re\left( i\xi\hat{\psi}_t\overline{\hat{\varphi}_t}\right) \notag\\
&-\frac{k_0bl}{k}Re \left( i\xi\hat{\psi}\overline{\left ( i\xi \hat{\omega } - l \hat{\varphi} \right )}\right)  + \frac{bl\gamma}{k} |\xi| |\hat{\psi}||\hat{\theta}_1|+C_5|\xi|^2|\hat{\theta}_2|^2,
\end{align}
where $C_5$ is a positive constant.
\end{lem}
\begin{proof}
Multiplying $(\ref{e2})$ by $-\overline{ \left ( i\xi \hat{\varphi} - \hat{\psi} -l \hat{\omega }  \right )}$ and taking real part,
\begin{align}\label{e11}
\frac{d}{dt}Re\left(-\rho_2\hat{\psi}_{t}\overline{ \left ( i\xi \hat{\varphi} - \hat{\psi} -l \hat{\omega }  \right )}\right) - Re\left(i\rho_2\xi\hat{\psi}_{t}\overline{\hat{\varphi}_t}\right)  &-\rho_2 |\hat{\psi}_t|^2 -Re \left(\rho_2 l\hat{\psi}_t \overline{\hat{\omega }_t}\right)-Re\left(b\xi ^{2}\hat{\psi }\overline{ \left ( i\xi \hat{\varphi} -\hat{\psi} -l \hat{\omega }  \right )}\right) \notag\\
&+ k| i\xi \hat{\varphi} - \hat{\psi} -l \hat{\omega }  |^2  -Re\left( i\gamma \xi  \hat {\theta}_2\overline{ \left ( i\xi \hat{\varphi} - \hat{\psi} -l \hat{\omega }  \right )}\right) =0.
\end{align}
First, suppose that 
\begin{equation}\label{swp}
\text{$\frac{\rho_1}{\rho_2} =\frac{k}{b}$ and $k=k_0$}.
\end{equation}
By $(\ref{e1})$ and $(\ref{swp})$,
\begin{align*}
Re\left(b\xi ^{2}\hat{\psi }\overline{ \left ( i\xi \hat{\varphi} - \hat{\psi} -l \hat{\omega }  \right )}\right) &=bRe\left(i\xi\hat{\psi }\overline{i\xi \left ( i\xi \hat{\varphi} -\hat{\psi} -l \hat{\omega }  \right )}\right) \\
&=\frac{\rho_1 b}{k}Re\left(i\xi\hat{\psi}\overline{\hat{\varphi} _{tt}}\right)  -blRe \left(i\xi\hat{\psi} \overline{\left ( i\xi \hat{\omega } - l \hat{\varphi} \right )}\right) +\frac{bl\gamma}{k}Re\left(i\xi\hat{\psi}\overline{\hat{\theta}_1}\right) \\
&=\rho_2\frac{d}{dt} Re\left(i\xi\hat{\psi}\overline{\hat{\varphi} _{t}}\right) -\rho_2 Re\left(i\xi\hat{\psi}_t\overline{\hat{\varphi} _{t}}\right) -blRe \left(i\xi\hat{\psi} \overline{\left ( i\xi \hat{\omega } - l \hat{\varphi} \right )}\right) +\frac{bl\gamma}{k}Re\left(i\xi\hat{\psi}\overline{\hat{\theta}_1}\right).
\end{align*}
Substituting in $(\ref{e11})$, we have
\begin{align*}
\frac{d}{dt}J_3(\xi,t)+k| i\xi \hat{\varphi} - \hat{\psi} -l \hat{\omega }  |^2 \leq &\rho_2 |\hat{\psi}_t|^2 + Re \left(\rho_2 l\hat{\psi}_t \overline{\hat{\omega }_t}\right)-blRe \left(i\xi\hat{\psi} \overline{\left ( i\xi \hat{\omega } - l \hat{\varphi} \right )}\right) \\
&+\frac{bl\gamma}{k} |\xi| |\hat{\psi}||\hat{\theta}_1|+\gamma |\xi| |\hat {\theta}_2|| i\xi \hat{\varphi} - \hat{\psi} -l \hat{\omega }|.
\end{align*}
Applying Young inequality, $(\ref{e12})$ follows. Now, suppose that 
\begin{equation}\label{dwp}
\frac{\rho_1}{\rho_2} \neq \frac{k}{b} \quad \text{or} \quad k\neq k_0.
\end{equation}
Proceeding as above, $(\ref{e1})$ implies that
\begin{align*}
Re\left(b\xi ^{2}\hat{\psi }\overline{ \left ( i\xi \hat{\varphi} - \hat{\psi} -l \hat{\omega }  \right )}\right) 
&=\frac{\rho_1 b}{k}\frac{d}{dt}Re\left(i\xi\hat{\psi}\overline{\hat{\varphi} _{t}}\right) - \frac{\rho_1 b}{k}Re\left(i\xi\hat{\psi}_t\overline{\hat{\varphi} _{t}}\right)  -\frac{k_0}{k}blRe \left(i\xi\hat{\psi} \overline{\left ( i\xi \hat{\omega } - l \hat{\varphi} \right )}\right) \\
&+\frac{bl\gamma}{k}Re\left(i\xi\hat{\psi}\overline{\hat{\theta}_1}\right). 
\end{align*}
Substituting in $(\ref{e11})$ and applying Young inequality, we obtain $(\ref{e12'})$. 
\end{proof}

\begin{lem}\label{lem5}
Let $0<\varepsilon_1 < \frac{\rho_2l^2}{2\rho_1}$ and consider the functional
\begin{equation*}
J_4(\xi,t)=Re\left(\frac{\rho_2^2l^2}{\rho_1}\hat{\psi}_t\overline{\hat{\psi}}-\rho_2l\hat{\omega}_t\overline{\hat{\psi}}\right).
\end{equation*}
Then,
\begin{align}\label{e13'}
\frac{d}{dt}J_4(\xi,t)+b\left(\frac{\rho_2l^2}{\rho_1}-\frac{\varepsilon_1}{2}\right) \xi^2|\hat{\psi}|^2 \leq  \frac{\rho_2^2l^2}{\rho_1}|\hat{\psi}_t|^2 - \rho_2lRe\left(\overline{\hat{\psi}}_t\hat{\omega}_t\right) + \frac{\rho_2k_0l}{\rho_1}Re&\left( i\xi \hat{\psi}\overline{\left ( i\xi \hat{\omega } - l \hat{\varphi} \right )}\right) \notag \\
&+ C(\varepsilon_1)\left(|\hat{\theta}_1|^2+|\hat{\theta}_2|^2\right),
\end{align}
where $ C(\varepsilon_1)$ is a positive constant.
\end{lem}


\begin{proof}
Multiplying $(\ref{e2})$ by $\overline{ \hat{\psi}}$ and taking real part,
\begin{align}\label{ee12}
\frac{d}{dt}Re\left(\rho_2\hat{\psi}_{t}\overline{ \hat{\psi}}\right)-\rho_2|\hat{\psi}_t|^2+b\xi ^{2}|\hat{\psi }|^2- Re\left( k\overline{ \hat{\psi}}\left ( i\xi \hat{\varphi} - \hat{\psi} -l \hat{\omega }  \right ) \right)+Re \left( i\gamma \xi\overline{ \hat{\psi}}  \hat {\theta}_2\right) =0.
\end{align}
$(\ref{e3})$ implies that,
\begin{align*}
\frac{d}{dt}Re\left(\rho_2\hat{\psi}_{t}\overline{ \hat{\psi}}\right) - \frac{d}{dt}Re\left(\frac{\rho_1}{l}\overline{\hat{\psi}}\hat{\omega}_{t}\right)
+b\xi ^{2}|\hat{\psi }|^2 &= \rho_2|\hat{\psi}_t|^2-\frac{\rho_1}{l}Re\left(\overline{\hat{\psi}_t}\hat{\omega}_{t}\right) - \frac{k_0}{l}Re\left(i\xi\overline{\hat{\psi}} \left ( i\xi \hat{\omega } - l \hat{\varphi} \right )\right) \\
&+\frac{\gamma}{l}Re\left(i\xi\overline{\hat{\psi}} \hat{\theta}_1\right) - Re \left( i\gamma \xi\overline{ \hat{\psi}}  \hat {\theta}_2\right).
\end{align*}
Multiplying above by $\frac{\rho_2 l^2}{\rho_1}$, we have
\begin{align*}
\frac{d}{dt}J_4(\xi,t)+\frac{b\rho_2 l^2}{\rho_1}\xi ^{2}|\hat{\psi }|^2 &= \frac{\rho_2^2 l^2}{\rho_1}|\hat{\psi}_t|^2-\rho_2lRe\left(\overline{\hat{\psi}_t}\hat{\omega}_{t}\right) - \frac{\rho_2k_0l}{\rho_1}Re\left(i\xi\overline{\hat{\psi}} \left ( i\xi \hat{\omega } - l \hat{\varphi} \right )\right) \\
&+\frac{\gamma\rho_2 l}{\rho_1}Re\left(i\xi\overline{\hat{\psi}} \hat{\theta}_1\right) - \frac{\gamma\rho_2 l^2}{\rho_1}Re \left( i\xi\overline{ \hat{\psi}}  \hat {\theta}_2\right).
\end{align*}
Applying Young inequality and using the wave speeds of propagation, we obtain $(\ref{e13'})$.
\end{proof}

\begin{lem}\label{lem6}
Let $0<\varepsilon_1 < \dfrac{\rho_2l^2}{2\rho_1}$ and consider $K(\xi,t)=J_3(\xi,t)+J_4(\xi,t)$. If $\frac{\rho_1}{\rho_2} =\frac{k}{b}$ and $k=k_0$, then 
\begin{align*}
\frac{d}{dt}K(\xi,t)+\left(\frac{\rho_2l^2}{\rho_1}-\varepsilon_1\right) b\xi^2|\hat{\psi}|^2 + \frac{k}{2}|i\xi \hat{\varphi} - \hat{\psi} -l \hat{\omega }|^2 \leq \rho_2s_1|\hat{\psi}_t|^2 + C(\varepsilon_1)|\hat{\theta}_1|^2+  C(\varepsilon_1)(1+\xi^2)|\hat{\theta}_2|^2.
\end{align*} 
Moreover, if $\frac{\rho_1}{\rho_2} \neq \frac{k}{b}$  or $k\neq k_0$, then
\begin{align*}
\frac{d}{dt}K(\xi,t)&+\left(\frac{\rho_2l^2}{\rho_1}-\varepsilon_1\right) b\xi^2|\hat{\psi}|^2 + \frac{k}{2}|i\xi \hat{\varphi} - \hat{\psi} -l \hat{\omega }|^2 \notag\\
\leq  &\rho_2s_1|\hat{\psi}_t|^2 + \left(\frac{\rho_2}{\rho_1}-\frac{b}{k}\right)k_0l Re\left(i\xi \hat{\psi}(\overline{i\xi\hat{\omega}-l\hat{\varphi}})\right) + \left(\rho_2-\frac{b\rho_1}{k}\right) Re\left(i\xi \hat{\psi}_t\overline{\hat{\varphi}}_t\right) \notag \\
&+ 
C(\varepsilon_1)|\hat{\theta}_1|^2 +  C(\varepsilon_1)(1+\xi^2)|\hat{\theta}_2|^2,
\end{align*}
where $s_1= \frac{\rho_2 l^2}{\rho_1}+1$.
\end{lem}

\begin{proof}
It follows from Lemmas $\ref{lem4}$ and $\ref{lem5}$, applying Young inequality.
\end{proof}

For next lemma, it is necessary to observe that 
\begin{align}
\left(i\xi\hat{\varphi} - \hat{\psi} - l \hat{\omega}\right)_t - i\xi\hat{\varphi}_t + \hat{\psi}_t +l \hat{\omega}_t &= 0 \label{e16} \\
\left(i\xi\hat{\omega} - l \hat{\varphi}\right)_t - i\xi\hat{\omega}_t +l\hat{\varphi}_t &= 0 \label{e17}
\end{align}


\begin{lem}\label{lem7}
Consider the functional, 
\begin{equation*}
H(\xi,t)=\rho_1 Re\left( \left(i\xi \hat{\varphi}- \hat{\psi} -l \hat{\omega }\right)\overline{\hat{\omega}_t}\right)+\rho_1 Re\left(\left(i\xi \hat{\omega } - l \hat{\varphi}\right)\overline{\hat{\varphi}_t}\right).
\end{equation*}
If $\frac{\rho_1}{\rho_2} =\frac{k}{b}$ and $k=k_0$, then
\begin{align}\label{e15}
\frac{d}{dt}H(\xi,t)+\rho_1 l |\hat{\varphi}_t|^2+\frac{\rho_1 l}{2}|\hat{\omega}_t|^2 \leq \frac{\rho_2 k}{2 b l}|\hat{\psi}_t|^2 +\frac{3kl}{2}|i\xi \hat{\varphi} -\hat{\psi} -l \hat{\omega }|^2+\frac{3k_0l}{2}|i\xi \hat{\omega } - l \hat{\varphi}|^2 +C_6(1+\xi^2)|\hat{\theta}_1|^2.
\end{align}
Moreover, if $\frac{\rho_1}{\rho_2} \neq\frac{k}{b}$ or $k\neq k_0$,
\begin{align}\label{e15'}
\frac{d}{dt}H(\xi,t)+\rho_1 l |\hat{\varphi}_t|^2+\frac{\rho_1 l}{2}|\hat{\omega}_t|^2 \leq &\frac{\rho_1}{2l}|\hat{\psi}_t|^2 +C_1(k,k_0)|i\xi \hat{\varphi} -\hat{\psi} -l \hat{\omega }|^2+C_2(k,k_0)(1+\xi^2)|i\xi \hat{\omega } - l \hat{\varphi}|^2 \notag\\
&+C_6(1+\xi^2)|\hat{\theta}_1|^2,
\end{align}
where $C_1(k,k_0)$ and $C_6$ are positive constants.
\end{lem}
\begin{proof}
Multiplying $(\ref{e1})$ by $\overline{\left(i\xi \hat{\omega } - l \hat{\varphi}\right)}$, $(\ref{e17})$ by $\rho_1\overline{\hat{\varphi}_t}$, adding these equalities and taking the real part, we obtain
\begin{align}\label{e20}
\frac{d}{dt}Re\left(\rho_1\left(i\xi\hat{\omega} - l \hat{\varphi}\right)\overline{\hat{\varphi}_t}\right) -\rho_1Re\left( i\xi\hat{\omega}_t \overline{\hat{\varphi}_t}\right)+\rho_1l|\hat{\varphi}_t|^2 &-kRe\left(i\xi \left ( i\xi \hat{\varphi} - \hat{\psi} -l \hat{\omega }  \right )\overline{\left(i\xi \hat{\omega } - l \hat{\varphi}\right)}\right) \notag\\
&-k_0l| i\xi \hat{\omega } - l \hat{\varphi}|^2 +l\gamma Re\left(\hat{\theta}_1\overline{\left(i\xi \hat{\omega } - l \hat{\varphi}\right)}\right) = 0.
\end{align}
Multiplying $(\ref{e3})$ by $\overline{\left ( i\xi \hat{\varphi} - \hat{\psi} -l \hat{\omega }\right)}$, $(\ref{e16})$ by $\rho_1\overline{\hat{\omega}_t}$, adding and  taking the real part, 
\begin{align}\label{e23}
\frac{d}{dt}Re\left(\rho_1\left(i\xi\hat{\varphi} - \hat{\psi} - l \hat{\omega}\right)\overline{\hat{\omega}_t}\right) &-\rho_1Re\left( i\xi\hat{\varphi}_t\overline{\hat{\omega}_t}\right)+\rho_1 Re\left( \hat{\psi}_t\overline{\hat{\omega}_t}\right) +\rho_1l |\hat{\omega}_t|^2 - kl| i\xi \hat{\varphi} - \hat{\psi} -l \hat{\omega }|^2\notag \\
& -k_0Re\left(i\xi \left ( i\xi \hat{\omega } - l \hat{\varphi} \right )\overline{\left ( i\xi \hat{\varphi} - \hat{\psi} -l \hat{\omega }\right)}\right) 
 + \gamma Re\left(i \xi \hat{\theta}_1\overline{\left ( i\xi \hat{\varphi} - \hat{\psi} -l \hat{\omega }\right)}\right)=0.
\end{align}
Adding $(\ref{e20})$ and $(\ref{e23})$,
\begin{align*}
\frac{d}{dt}H(\xi,t) +\rho_1l |\hat{\varphi}_t|^2 +\rho_1l |\hat{\omega}_t|^2 \leq &kl| i\xi \hat{\varphi} - \hat{\psi} -l \hat{\omega }|^2 + k_0l| i\xi \hat{\omega } - l \hat{\varphi}|^2 +\gamma |\xi||\hat{\theta}_1|| i\xi \hat{\varphi} - \hat{\psi} -l \hat{\omega }| + \gamma l |\hat{\theta}_1|| i\xi \hat{\omega } - l \hat{\varphi}| \\
&+|k-k_0||\xi||i\xi\hat{\omega}-l\hat{\varphi}||i\xi\hat{\varphi}-\hat{\psi}-l\hat{\omega}|+\rho_1|\hat{\psi}_t||\hat{\omega}_t|.
\end{align*}
Applying Young inequality and using the wave propagation properties, $(\ref{e15})$ and $(\ref{e15'})$ follow.
\end{proof}
In order to prove our main result, we need to establish a functional equivalent to the energy \eqref{energytypeII} in a polynomial sense. In particular, this kind of functional gives some dissipative terms of the vector solution. First, let us define:
\begin{equation}
\LL_1(\xi,t) = \left\lbrace\begin{tabular}{l l}
$J_1+\varepsilon_2 \xi^2 K+\xi^2J_2+\varepsilon_3\xi^2H$, & if $\frac{\rho_1}{\rho_2} =\frac{k}{b}$ and $k=k_0$, \\
$\frac{\xi^2}{(1+\xi^2+\xi^4)}\left( \lambda_1  \varepsilon_3 J_1+ \frac{\xi^2}{(1+\xi^2+\xi^4)}\left(\varepsilon_3\lambda_2  K + J_2+\varepsilon_3H \right) \right)$, & if $\frac{\rho_1}{\rho_2} \neq \frac{k}{b}$ or $k\neq k_0$.
\end{tabular}\right.
\end{equation}
where $\lambda_1, \lambda_2, \varepsilon_2, \varepsilon_3$ are positive constants to be fixed later. 
\begin{prop}\label{lem8}
There exists constants $M, M' > 0$ such that, if $\frac{\rho_1}{\rho_2} =\frac{k}{b}$ and $k=k_0$, then
\begin{align}\label{e24}
\frac{d}{dt}\LL_1(\xi,t) &+M\xi^2 \left\lbrace b \xi^2|\hat{\psi}|^2 + k|i\xi \hat{\varphi} - \hat{\psi} -l \hat{\omega }|^2 +\rho_2|\hat{\psi}_t|^2+k_0|i\xi \hat{\omega } - l \hat{\varphi}|^2 + \rho_1|\hat{\varphi}_t|^2+\rho_1|\hat{\omega}_t|^2\right\rbrace \notag\\
&\leq C(\varepsilon_1,\varepsilon_2,\varepsilon_3)(1+\xi^2+\xi^4+\xi^6)\xi^2|\hat{\theta}_1|^2 +C(\varepsilon_1,\varepsilon_2)(1+\xi^2+\xi^4)|\hat{\theta}_2|^2.
\end{align}
Moreover, if $\frac{\rho_1}{\rho_2} \neq \frac{k}{b}$ or $k\neq k_0$, we obtain
\begin{align}\label{e24'}
\frac{d}{dt}\LL_1(\xi,t) &+M'\frac{\xi^4}{(1+\xi^2+\xi^4)^2} \left\lbrace b \xi^2|\hat{\psi}|^2 + k|i\xi \hat{\varphi} - \hat{\psi} -l \hat{\omega }|^2 +\rho_2|\hat{\psi}_t|^2+k_0|i\xi \hat{\omega } - l \hat{\varphi}|^2 + \rho_1|\hat{\varphi}_t|^2+\rho_1|\hat{\omega}_t|^2\right\rbrace \notag\\
&\leq C(\varepsilon_1,\varepsilon_3,\lambda_1,\lambda_2)(1+\xi^2)\xi^2\left(|\hat{\theta}_1|^2 +|\hat{\theta}_2|^2 \right)
\end{align}
\end{prop}

\begin{proof}
First, we suppose that $\frac{\rho_1}{\rho_2} =\frac{k}{b}$ and $k=k_0$. By using  Lemmas $\ref{lem2}$ and $\ref{lem6}$, we have 
\begin{align*}
\frac{d}{dt}\left\lbrace J_1(\xi,t)+\varepsilon_2  \xi^2 K(\xi,t)\right\rbrace +b\varepsilon_2\left(\frac{\rho_2l^2}{\rho_1}-\varepsilon_1\right) \xi^4|\hat{\psi}|^2 &+ k\frac{\varepsilon_2}{2}\xi^2|i\xi \hat{\varphi} - \hat{\psi} -l \hat{\omega }|^2 \\
+\rho_2\left(\frac{m_2}{2}-s_1\varepsilon_2\right)\xi^2|\hat{\psi}_t|^2
&\leq b|\xi|^3|\hat{\psi}||\hat{\theta}_2|+ k|\xi||\hat{\theta}_2| | i\xi \hat{\varphi} + \hat{\psi} +l \hat{\omega } |\\
&+C(\varepsilon_1, \varepsilon_2)\xi^2|\hat{\theta}_1|^2+  C(\varepsilon_1,\varepsilon_2)(1+\xi^2)\xi^2|\hat{\theta}_2|^2. 
\end{align*}
On the other hand, By Lemmas $\ref{lem3}$ and$\ref{lem7}$, it follows that 
\begin{align*}
\frac{d}{dt}\left\lbrace \xi^2J_2(\xi,t)+\varepsilon_3\xi^2H(\xi,t) \right\rbrace &+ k_0\delta\xi^2|i\xi \hat{\omega } - l \hat{\varphi}|^2 + 
\rho_1\varepsilon_3 l \xi^2|\hat{\varphi}_t|^2+\frac{\rho_1 \varepsilon_3 l}{2}\xi^2|\hat{\omega}_t|^2 \\
&\leq \frac{\rho_1lk_1}{m_1}|\xi|^4|\hat{\varphi}_t||\hat{\theta}_1|+\frac{\rho_1k_1}{m_1}|\xi|^5|\hat{\theta}_1||\hat{\omega}_t| +\frac{2kl}{m_1}|\xi|^3|\hat{\theta}_1||i\xi \hat{\varphi} -\hat{\psi} -l \hat{\omega }|\\
&+C_4(1+\xi^2) \xi^2 |\hat{\theta}_1|^2+\frac{3k_0\varepsilon_3 l}{2}\xi^2|i\xi \hat{\omega } - l \hat{\varphi}|^2 +\frac{3kl\varepsilon_3}{2}|\xi|^2|i\xi \hat{\varphi} -\hat{\psi} -l \hat{\omega }|^2 \\
&+\frac{\rho_2 \varepsilon_3 k}{2 b l}\xi^2|\hat{\psi}_t|^2+C(\varepsilon_3)(1+\xi^2)\xi^2|\hat{\theta}_1|^2.
\end{align*}
Adding and by using Young inequality, we obtain
\begin{align*}
\frac{d}{dt}\LL_1(\xi,t) +b\varepsilon_2\left(\frac{\rho_2l^2}{\rho_1}-2\varepsilon_1\right) \xi^4|\hat{\psi}|^2 &+ k\left(\frac{\varepsilon_2}{4}-2l\varepsilon_3\right)\xi^2|i\xi \hat{\varphi} - \hat{\psi} -l \hat{\omega }|^2 +\rho_2\left(\frac{m}{2}-s_1\varepsilon_2-\frac{k}{2bl}\varepsilon_3\right)\xi^2|\hat{\psi}_t|^2 \notag\\
&+k_0\left(\delta-\frac{3l}{2}\varepsilon_3\right)\xi^2|i\xi \hat{\omega } - l \hat{\varphi}|^2 + 
\frac{\rho_1l}{2}\varepsilon_3 |\xi|^2|\hat{\varphi}_t|^2+\frac{\rho_1l }{4}\varepsilon_3|\xi|^2|\hat{\omega}_t|^2 \notag \\
&\leq C(\varepsilon_1,\varepsilon_2,\varepsilon_3)(1+\xi^2+\xi^4+\xi^6)\xi^2|\hat{\theta}_1|^2 +C(\varepsilon_1,\varepsilon_2)(1+\xi^2+\xi^4)|\hat{\theta}_2|^2. 
\end{align*}
We choose our constants as follows, 
\[
\varepsilon_1 < \frac{\rho_2 l^2}{2 \rho_1}, \quad
\varepsilon_2 < \frac{m_2}{2s_1} \quad \text{and} \quad
\varepsilon_3 < \min\left\lbrace \frac{2\delta}{3l}, \frac{\varepsilon_2}{8l},\frac{2bl}{k}\left(\frac{m_2}{2}-s_1\varepsilon_2\right)\right\rbrace.
\]
Consequently, we deduce that there exist $M>0$, such that
\begin{align*}
\frac{d}{dt}\LL_1(\xi,t) &+M\xi^2 \left\lbrace b \xi^2|\hat{\psi}|^2 + k|i\xi \hat{\varphi} - \hat{\psi} -l \hat{\omega }|^2 +\rho_2|\hat{\psi}_t|^2+k_0|i\xi \hat{\omega } - l \hat{\varphi}|^2 + \rho_1|\hat{\varphi}_t|^2+\rho_1|\hat{\omega}_t|^2\right\rbrace \\
&\leq C(\varepsilon_1,\varepsilon_2,\varepsilon_3)(1+\xi^2+\xi^4+\xi^6)\xi^2|\hat{\theta}_1|^2 +C(\varepsilon_1,\varepsilon_2)(1+\xi^2+\xi^4)|\hat{\theta}_2|^2. 
\end{align*}
At last, we assume that $\frac{\rho_1}{\rho_2} \neq \frac{k}{b}$ or $k\neq k_0$. Consider the functional 
\[
P_1 =  \frac{\xi^2}{1+\xi^2+\xi^4} \lambda_1 \varepsilon_3 J_1(\xi,t)+\frac{\xi^4}{(1+\xi^2+\xi^4)^2}\lambda_2 \varepsilon_3 K(\xi,t).
\]
By Lemmas $\ref{lem2}$ and $\ref{lem6}$ and by using Young inequality, it follows that
\begin{align*}
\frac{d}{dt} P_1(\xi,t) &+\left(\frac{\rho_2l^2}{\rho_1}-3\varepsilon_1\right) \frac{\lambda_2\varepsilon_3b\xi^6}{(1+\xi^2+\xi^4)^2} |\hat{\psi}|^2 + \frac{\varepsilon_3\lambda_2k\xi^4}{4(1+\xi^2+\xi^4)^2}|i\xi \hat{\varphi} - \hat{\psi} -l \hat{\omega }|^2 \notag \\
&+\left(\frac{\lambda_1m_2}{2}-C(\varepsilon_4,\lambda_2)\right)\frac{\rho_2\varepsilon_3\xi^4}{(1+\xi^2+\xi^4)}|\hat{\psi}_t|^2 \\
\leq &\frac{\lambda_2\varepsilon_3\varepsilon_4 \xi^4|\hat{\varphi}_t|^2}{(1+\xi ^2+\xi^4)^2}  +  C(\varepsilon_1,\lambda_2)\varepsilon_3\frac{(1+\xi^2)\xi^4|i\xi\hat{\omega}-l\hat{\varphi}|^2}{(1+\xi^2+\xi^4)^2} +C(\varepsilon_1,\varepsilon_3,\lambda_1, \lambda_2)(1+ \xi^2)\xi^2\left(|\hat{\theta}_1|^2+ |\hat{\theta}_2|\right).
\end{align*}
In the last estimate we used the following inequalities:
\begin{equation}\label{ert}
\frac{1}{(1+\xi^2+\xi^4)}\leq 1, \quad \frac{1+\xi^4}{(1+\xi^2+\xi^4)}\leq 1, \quad \frac{\xi^2}{(1+\xi^2+\xi^4)}\leq 1.
\end{equation}
On the other hand, consider the functional  
\[
P_2 =  \frac{\xi^4}{(1+\xi^2+\xi^4)^2}\left( J_2(\xi,t)+\varepsilon_3 H(\xi,t)\right).
\]
By (\ref{e43}) in Lemma $\ref{lem3}$, Lemma $\ref{lem7}$, Young inequality and (\ref{ert}), we obtain
\begin{align*}
\frac{d}{dt}P_2 &+ k_0\left(\delta- C(k,k_0)\varepsilon_3\right)\frac{(1+\xi^2)\xi^4}{(1+\xi^2+\xi^4)^2}|i\xi \hat{\omega } - l \hat{\varphi}|^2 + 
\frac{\rho_1\varepsilon_3 l \xi^4}{2(1+\xi^2+\xi^4)^2}|\hat{\varphi}_t|^2+\frac{\rho_1\varepsilon_3 l\xi^4}{4(1+\xi^2+\xi^4)^2}|\hat{\omega}_t|^2 \notag \\
&\leq \frac{\rho_1 \varepsilon_3 \xi^4}{2l(1+\xi^2+\xi^4)}|\hat{\psi}_t|^2  +C(k,k_0)\varepsilon_3\frac{\xi^4}{(1+\xi^2+\xi^4)^2}|i\xi \hat{\varphi} -\hat{\psi} -l \hat{\omega }|^2 +C(\varepsilon_3)(1+\xi^2)\xi^2|\hat{\theta}_1|^2.
\end{align*}
Thus, Adding above estimates of $P_1$ and $P_2$, we get  
\begin{align*}
\frac{d}{dt}\LL_1(\xi,t) &+\left(\frac{\rho_2l^2}{\rho_1}-3\varepsilon_1\right)\frac{\lambda_2\varepsilon_3b\xi^6}{(1+\xi^2+\xi^4)^2}|\hat{\psi}|^2 + \left(\frac{\lambda_2}{4}-C(k,k_0)\right)\frac{\varepsilon_3k\xi^4}{(1+\xi^2+\xi^4)^2}|i\xi \hat{\varphi} - \hat{\psi} -l \hat{\omega }|^2  \\
&+\left(\frac{\lambda_1m_2}{2}-C(\varepsilon_4,\lambda_2)-\frac{\rho_1}{2l\rho_2}\right)\frac{\rho_2\varepsilon_3\xi^4}{(1+\xi^2+\xi^4)}|\hat{\psi}_t|^2  +\frac{\rho_1 \varepsilon_3l\xi^4}{4(1+\xi^2+\xi^4)^2}|\hat{\omega}_t|^2 \notag\\
&+\left(\delta-C(k,k_0,\varepsilon_1,\lambda_2)\varepsilon_3\right)\frac{k_0(1+\xi^2)\xi^4}{(1+\xi^2+\xi^4)^2}|i\xi \hat{\omega } - l \hat{\varphi}|^2 + 
\left(\frac{l}{2}-\frac{\varepsilon_4\lambda_2}{\rho_1}\right)\frac{\rho_1\varepsilon_3\xi^4}{(1+\xi^2+\xi^4)^2}|\hat{\varphi}_t|^2\notag \\
&\leq C(\varepsilon_1,\varepsilon_3,\lambda_1,\lambda_2)(1+\xi^2)\xi^2\left(|\hat{\theta}_1|^2 +|\hat{\theta}_2|^2\right)
\end{align*}
We choose our constants as follows:
\[
\varepsilon_1 < \frac{\rho_2 l^2}{3\rho_1}, \quad
\lambda_2 > 4C(k,k_0), \quad
\varepsilon_4 < \frac{\rho_1 l}{2\lambda_2}, \quad
\lambda_1 > \frac{\rho_1+2l\rho_2C(\varepsilon_4,\lambda_2) }{m_2l \rho_2} \quad \text{and} \quad
\varepsilon_3 < \frac{\delta}{C(k,k_0,\varepsilon_2,\lambda_2)}
\]
Consequently, by using (\ref{ert}), we deduce that there exist $M'>0$, such that
\begin{align*}
\frac{d}{dt}\LL_1(\xi,t) &+M'\frac{\xi^4}{(1+\xi^2+\xi^4)^2} \left\lbrace b \xi^2|\hat{\psi}|^2 + k|i\xi \hat{\varphi} - \hat{\psi} -l \hat{\omega }|^2 +\rho_2|\hat{\psi}_t|^2+k_0|i\xi \hat{\omega } - l \hat{\varphi}|^2 + \rho_1|\hat{\varphi}_t|^2+\rho_1|\hat{\omega}_t|^2\right\rbrace \\
&\leq C(\varepsilon_1,\varepsilon_3,\lambda_1,\lambda_2)(1+\xi^2)\xi^2\left(|\hat{\theta}_1|^2+|\hat{\theta}_2|^2 \right)
\end{align*}
\end{proof}

With the Proposition \ref{lem8} in hands and making some appropriate combinations, we build a Lyapunov functional $\LL$, which  plays a crucial role in the proof of our main result. 
\begin{thm}\label{teo1}
For any $t \geq 0$ and $\xi \in \R$, we obtain the following decay rates
\begin{equation}\label{eq27'}
\hat{E}(\xi,t) \leq \left\lbrace\begin{tabular}{l l}
$C \hat{E}(\xi,0) e^{-\beta s_1(\xi)t}$, & if $\frac{\rho_1}{\rho_2} =\frac{k}{b}$ and $k=k_0$, \\
$C' \hat{E}(\xi,0) e^{-\beta' s_2(\xi)t}$, & if $\frac{\rho_1}{\rho_2} \neq\frac{k}{b}$ or $k\neq k_0,$
\end{tabular}\right.
\end{equation}
where $C, \beta, C',\beta'$ are positive constants and 
\begin{equation*}
s_1(\xi)=\dfrac{\xi^4}{\left(1+\xi^2+\xi^4+\xi^6+\xi^8\right)}, \quad s_2(\xi)=\dfrac{\xi^4}{\left(1+\xi^2\right)\left(1+\xi^2+\xi^4\right)^2}.
\end{equation*}
\end{thm}

\begin{proof}
Consider the following Lyapunov functional:
\begin{equation}\label{f1}
\LL(\xi,t) = \left\lbrace\begin{tabular}{l l}
$\xi ^2\LL_1(\xi,t)+N(1+\xi^2+\xi^4+\xi^6+\xi^8)\hat{E}(\xi,t)$, & if $\frac{\rho_1}{\rho_2} =\frac{k}{b}$ and $k=k_0$, \\
$\LL_1(\xi,t)+N'(1+\xi^2)\hat{E}(\xi,t)$, & if $\frac{\rho_1}{\rho_2} \neq \frac{k}{b}$ or $k\neq k_0$.
\end{tabular}\right.
\end{equation}
where $N$ and $N'$ are positive constants to be fixed later.
First, we suppose that $\frac{\rho_1}{\rho_2} =\frac{k}{b}$ and $k=k_0$.  By Lemma $\ref{lem1}$ and Proposition $\ref{lem8}$,
\begin{align*}
\frac{d}{dt}\LL(\xi,t) &\leq - M\xi^4 \left\lbrace b \xi^2|\hat{\psi}|^2 + k|i\xi \hat{\varphi} - \hat{\psi} -l \hat{\omega }|^2 +\rho_2|\hat{\psi}_t|^2+k_0|i\xi \hat{\omega } - l \hat{\varphi}|^2 + \rho_1|\hat{\varphi}_t|^2+\rho_1|\hat{\omega}_t|^2\right\rbrace \\
&-\left(2\gamma \eta N-C(\varepsilon_1,\varepsilon_2,\varepsilon_3)\right)(1+\xi^2+\xi^4+\xi^6+\xi^8)\xi^2\left(|\hat{\theta}_1|^2 +|\hat{\theta}_2|^2 \right),
\end{align*}
where $\eta = \min\{\frac{k_1}{m_1},\frac{k_2}{m_2}\}$. On the other hand, by definition of $\LL_1$, there exist $M_1 >0$, such that
\[
|\xi^2\LL_1(\xi,t)| \leq M_1 (1+\xi^2+\xi^4+\xi^6) \hat{E}(\xi,t).
\]
It follows that 
\begin{equation}\label{e29}
(N-M_1)(1+\xi^2+\xi^4+\xi^6+\xi^8)\hat{E}(\xi,t) \leq \LL(\xi,t)\leq (N+M_1)(1+\xi^2+\xi^4+\xi^6+\xi^8) \hat{E}(\xi,t).
\end{equation}
Choosing $N > \max \left(M_1, \dfrac{C(\varepsilon_1,\varepsilon_2,\varepsilon_3)}{2\gamma \eta}\right)$ and by using 
\[
(1+\xi^2+\xi^4+\xi^6+\xi^8) \geq \xi^2,
\]
there exist $M_2>0$, such that 
\begin{equation}\label{e28}
\frac{d}{dt}\LL(\xi,t) \leq - M_2\xi^4 \hat{E}(\xi,t).
\end{equation}
Note that, $(\ref{e29})$ implies that
\begin{equation}\label{e30}
\frac{d}{dt}\LL(\xi,t) \leq - \beta\frac{\xi^4}{(1+\xi^2+\xi^4+\xi^6+\xi^8)} \LL(\xi,t),
\end{equation}
where $\beta=\dfrac{M_2}{N+M_1}$. By using Gronwall inequality, it follows that
\begin{equation}\label{e31}
\LL(\xi,t) \leq \LL(\xi,0) e^{-\beta s_1(\xi)t}.
\end{equation}
Now, from $(\ref{e29})$, it yields that 
$$\hat{E}(\xi,t)  \leq C \hat{E}(\xi,0) e^{-\beta s_1(\xi) t},$$
where $C= \frac{N+M_1}{N-M_1} > 0$.\\
\vglue 0.3cm

\noindent At last, we assume that $\frac{\rho_1}{\rho_2} \neq \frac{k}{b}$ or $k\neq k_0$. By Lemma $\ref{lem1}$ and Proposition  $\ref{lem8}$,
\begin{align*}
\frac{d}{dt}\LL(\xi,t) &\leq - M'\frac{\xi^4}{(1+\xi^2+\xi ^4)^2} \left\lbrace b \xi^2|\hat{\psi}|^2 + k|i\xi \hat{\varphi} - \hat{\psi} -l \hat{\omega }|^2 +\rho_2|\hat{\psi}_t|^2+k_0|i\xi \hat{\omega } - l \hat{\varphi}|^2 + \rho_1|\hat{\varphi}_t|^2+\rho_1|\hat{\omega}_t|^2\right\rbrace \\
&-\left(2\gamma \eta N'-C(\varepsilon_1,\varepsilon_3,\lambda_1,\lambda_2)\right)(1+\xi^2)\xi^2\left(|\hat{\theta}_1|^2 +|\hat{\theta}_2|^2 \right)
\end{align*}
On the other hand, by definition of $\LL_2$ and by using Young inequality, there exist $M'_1 >0$ such that
\begin{equation}\label{e29'}
(N'-M'_1)(1+\xi^2)\hat{E}(\xi,t) \leq \LL(\xi,t)\leq (N'+M'_1)(1+\xi^2) \hat{E}(\xi,t).
\end{equation}
Choosing $N' > \max \left(M'_1, \dfrac{C(\varepsilon_1,\varepsilon_3,\lambda_1,\lambda_2)}{2\gamma \eta}\right)$ and by using 
\[\
(1+\xi^2) \geq \frac{\xi^2}{(1+\xi^2+\xi^4)^2}, 
\]
there exist $M'_2>0$, such that 
\begin{equation}\label{e28'}
\frac{d}{dt}\LL(\xi,t) \leq - M'_2\frac{\xi^4}{(1+\xi^2+\xi^4)^2} \hat{E}(\xi,t).
\end{equation}
From $(\ref{e29'})$, we get 
\begin{equation}\label{e30'}
\frac{d}{dt}\LL(\xi,t) \leq - \beta'\frac{\xi^4}{(1+\xi^2)(1+\xi^2+\xi^4)^2}\LL(\xi,t),
\end{equation}
where $\beta'=\dfrac{M'_2}{N'+M'_1}$. By using Gronwall inequality, we conclude that 
\begin{equation}\label{e31'}
\LL(\xi,t) \leq \LL(\xi,0) e^{-\beta' s_2(\xi)t}.
\end{equation}
The last inequality together with $(\ref{e29'})$ leads to the second inequality of theorem, which completes the proof.
\end{proof}


\subsection{Thermoelastic Bresse system of type III}

In this subsection, we establish decay rates for the Fourier image of the solutions of Thermoelastic Bresse system of Type III. Taking Fourier Transform in (\ref{ee2}), we obtain the following ODE system:
\begin{align}
\rho_1\hat{\varphi} _{tt}-ik\xi \left ( i\xi \hat{\varphi} - \hat{\psi} -l \hat{\omega }  \right ) -k_0l\left ( i\xi \hat{\omega } - l \hat{\varphi} \right ) +l\gamma \hat{\theta}_{1t} &= 0 \quad \text{in $\R \times (0,\infty)$}, \label{eq1}\\
\rho_2\hat{\psi}_{tt}+b\xi ^{2}\hat{\psi }- k\left ( i\xi \hat{\varphi} - \hat{\psi} -l \hat{\omega }  \right )+i\gamma \xi  \hat {\theta}_{2t}&=0 \quad \text{in $\R \times (0,\infty)$},\label{eq2}\\
\rho_1\hat{\omega}_{tt}-ik_0\xi \left ( i\xi \hat{\omega } - l \hat{\varphi} \right )- kl\left ( i\xi \hat{\varphi} - \hat{\psi} -l \hat{\omega }  \right )+ i\gamma \xi \hat{\theta}_{1t}&=0 \quad \text{in $\R \times (0,\infty)$},\label{eq3}\\
\hat{\theta}_{1tt} +k_1\xi ^{2}\hat{\theta}_1+\alpha_1\xi ^{2}\hat{\theta}_{1t}+m_1 \left ( i\xi \hat{\omega } - l \hat{\varphi} \right )_t&=0 \quad \text{in $\R \times (0,\infty)$},\label{eq4} \\
\hat{\theta}_{2tt} +k_2\xi ^{2}\hat{\theta}_2+\alpha_2\xi ^{2}\hat{\theta}_{2t}+im_2 \xi \hat{\psi}_{t} &=0 \quad \text{in $\R \times (0,\infty)$}.\label{eq5}
\end{align}
The energy functional associated to the above system is defined as:
\begin{multline}\label{energytypeIII}
\hat{\E}\left ( \xi,t  \right )=\rho_1|\hat{\varphi}_{t} |^{2}+\rho_2|\hat{\psi}_{t}|^{2}+\rho_1 |\hat{\omega}_{t}|^{2}+\frac{\gamma}{m_1}|\hat{\theta}_{1t} |^{2}+\frac{k_1\gamma}{m_1}\xi^2|\hat{\theta}_1 |^{2}+\frac{\gamma}{m_2}|\hat{\theta}_{2t} |^{2}+\frac{k_2\gamma}{m_2}\xi^2|\hat{\theta}_2 |^{2}+b| \xi|^{2}|\hat{\psi}|^{2} \\
+k|i\xi \hat{\varphi} - \hat{\psi} -l \hat{\omega }|^{2} + k_0|i\xi \hat{\omega } - l \hat{\varphi}|^{2}
\end{multline}

\begin{lem}\label{lemq1}
Let $\hat{\E}$ the energy functional associated to the system \eqref{eq1}-\eqref{eq5}. Then, 
\begin{gather}\label{eq7}
\frac{d}{dt}\hat{\E}(\xi,t)=-2\gamma\xi^2\left(\frac{\alpha_1}{m_1}|\hat{\theta}_{1t}|^2+\frac{\alpha_2}{m_2}|\hat{\theta}_{2t}|^2\right).
\end{gather}
\end{lem}
\begin{proof}
Multiplying $(\ref{eq1})$ by $\overline{\hat{\varphi}}_t$, $(\ref{eq2})$ by $\overline{\hat{\psi}}_t$, $(\ref{eq3})$ by $\overline{\hat{\omega}}_t$, $(\ref{eq4})$ by $\frac{\gamma}{m_1}\overline{\hat{\theta}_{1t}}$, $(\ref{eq5})$ by $\frac{\gamma}{m_2}\overline{\hat{\theta}_{2t}}$,
adding these iqualities and taking the real part, (\ref{eq7}) follows.
\end{proof}
In order to establish the main result of this subsection and based in the approach done in the previous subsection, we establish the following  lemmas:
\begin{lem}\label{lem9}
The functional
\begin{equation*}
\J_1(\xi,t)=Re\left( i\rho_2\xi \hat{\psi}_t\overline{\hat{\theta}_{2t}}\right)+Re\left( ik_2\rho_2\xi^3 \hat{\psi}\overline{\hat{\theta}_{2}}\right),
\end{equation*}
satisfies
\begin{align}\label{eq9}
\frac{d}{dt}\J_1(\xi,t) + \frac{m_2\rho_2}{2}\xi^2|\hat{\psi}_t|^2 \leq k_2\rho_2|\xi|^3|\hat{\psi}||\hat{\theta}_{2t}| +b|\xi|^3|\hat{\psi}||\hat{\theta}_{2t}| + k|\xi||\hat{\theta}_{2t}| | i\xi \hat{\varphi} - \hat{\psi} -l \hat{\omega }|+  C_1(1+\xi^2)\xi^2|\hat{\theta}_{2t}|^2,
\end{align}
where $C_1$ is a positive constant.
\end{lem}
\begin{proof}
Multiplying $(\ref{eq5})$ by $-i\rho_2\xi \overline{\hat{\psi}_t}$ and taking real part, we obtain
\begin{align*}
\frac{d}{dt}Re\left(-i\rho_2\xi \overline{\hat{\psi}_t}\hat{\theta}_{2t}\right)+Re\left(i\rho_2\xi \overline{\hat{\psi}_{tt}}\hat{\theta}_{2t}\right) -\frac{d}{dt}Re\left(ik_2\rho_2\xi^3 \overline{\hat{\psi}}\hat{\theta}_2\right)+Re\left(ik_2\rho_2\xi^3 \overline{\hat{\psi}}\hat{\theta}_{2t}\right)& \\
-Re\left(i\alpha_2\rho_2\xi^3 \overline{\hat{\psi}_t}\hat{\theta}_{2t}\right)+m_2\rho_2 \xi^2 |\hat{\psi}_{t}|^2 &=0.
\end{align*}
By $(\ref{eq2})$, it follows that
\begin{align*}
\frac{d}{dt}\J_1(\xi,t)+m_2\rho_2 \xi^2 |\hat{\psi}_{t}|^2&
\leq k_2\rho_2|\xi|^3 |\hat{\psi}||\hat{\theta}_{2t}|+\alpha_2\rho_2|\xi|^3 |\hat{\psi}_t||\hat{\theta}_{2t}|+b|\xi|^3 |\hat{\psi}| |\hat{\theta}_{2t}| + k|\xi||\hat{\theta}_{2t}|| i\xi \hat{\varphi} - \hat{\psi} -l \hat{\omega } | \\
&+\gamma \xi^2 |\hat {\theta}_{2t}|^2.
\end{align*}
Applying Young inequality, $(\ref{eq9})$ follows.
\end{proof}

\begin{lem}\label{lem10'}
The functional
\begin{equation*}
\T_1(\xi,t)=Re\left(-\rho_1\hat{\varphi}_t \overline{\left(i\xi \hat{\omega } - l \hat{\varphi}\right)} - \frac{\rho_1}{m_1}\hat{\varphi}_t\overline{\hat{\theta}_{1t}}\right),
\end{equation*}
satisfies
\begin{align}\label{eq10}
\frac{d}{dt}\T_1(\xi,t) +\frac{k_0l}{2}|i\xi \hat{\omega } - l \hat{\varphi}|^2 \leq &\frac{\alpha_1\rho_1}{m_1}|\xi|^2|\hat{\varphi}_t||\hat{\theta}_{1t}|-Re(ik\xi(i\xi\hat{\varphi}-\hat{\psi}-l\hat{\omega})\overline{\left(i\xi \hat{\omega } - l \hat{\varphi}\right)}) \notag \\
&-\frac{k}{m_1}Re(i\xi\overline{\hat{\theta}_{1t}}(i\xi\hat{\varphi}-\hat{\psi}-l\hat{\omega})) +\frac{\rho_1k_1}{m_1}Re\left(\xi^2\hat{\varphi}_t\overline{\hat{\theta}}_{1}\right) +C_2 |\hat{\theta}_{1t}|^2,
\end{align}
where $C_2$ is a positive constant.
\end{lem}

\begin{proof}
Multiplying (\ref{eq1}) by $-\overline{\left(i\xi \hat{\omega } - l \hat{\varphi}\right)}$ and taking real part, we have
\begin{align*}
\frac{d}{dt} Re(-\rho_1\hat{\varphi} _{t}\overline{\left(i\xi \hat{\omega } - l \hat{\varphi}\right)})+Re(\rho_1\hat{\varphi} _{t}\overline{\left(i\xi \hat{\omega } - l \hat{\varphi}\right)}_t)&+Re(ik\xi \left ( i\xi \hat{\varphi} - \hat{\psi} -l \hat{\omega }  \right)\overline{\left(i\xi \hat{\omega } - l \hat{\varphi}\right)}) \notag\\
&+k_0l|i\xi \hat{\omega } - l \hat{\varphi}|^2 -Re(l\gamma \hat{\theta}_{1t}\overline{\left(i\xi \hat{\omega } - l \hat{\varphi}\right)})= 0,
\end{align*}
by using (\ref{eq4}), we have
\begin{align*}
\frac{d Re}{dt}(-\rho_1\hat{\varphi} _{t}\overline{\left(i\xi \hat{\omega } - l \hat{\varphi}\right)})&-\frac{\rho_1}{m_1}\frac{d}{dt}Re\left(\hat{\varphi}_t\overline{\hat{\theta}}_{1t}\right)+\frac{\rho_1}{m_1}Re\left(\hat{\varphi}_{tt}\overline{\hat{\theta}}_{1t}\right) -\frac{\rho_1k_1}{m_1}Re\left(\xi^2\hat{\varphi}_t\overline{\hat{\theta}}_{1}\right)-\frac{\alpha_1\rho_1}{m_1}Re\left(\xi^2\hat{\varphi}_t\overline{\hat{\theta}}_{1t}\right) \notag \\
&+Re \left( ik\xi \left ( i\xi \hat{\varphi} - \hat{\psi} -l \hat{\omega }  \right)\overline{\left(i\xi \hat{\omega } - l \hat{\varphi}\right)}\right) +k_0l|i\xi \hat{\omega } - l \hat{\varphi}|^2 -Re(l\gamma \hat{\theta}_{1t}\overline{\left(i\xi \hat{\omega } - l \hat{\varphi}\right)})= 0.
\end{align*}
Note that, (\ref{eq1}) implies that
\begin{align*}
\frac{d}{dt}\T_1(\xi,t) &+ \frac{k}{m_1}Re\left( i\xi\overline{\hat{\theta}}_{1t}\left ( i\xi \hat{\varphi} - \hat{\psi} -l \hat{\omega }  \right )\right)  +Re \left( ik\xi \left ( i\xi \hat{\varphi} - \hat{\psi} -l \hat{\omega }  \right)\overline{\left(i\xi \hat{\omega } - l \hat{\varphi}\right)}\right)+k_0l|i\xi \hat{\omega } - l \hat{\varphi}|^2  \\
&-\frac{\rho_1k_1}{m_1}Re\left(\xi^2\hat{\varphi}_t\overline{\hat{\theta}}_{1}\right)-\frac{\alpha_1\rho_1}{m_1}Re\left(\xi^2\hat{\varphi}_t\overline{\hat{\theta}}_{1t}\right) \leq  \frac{l\gamma}{m_1} |\hat{\theta}_{1t} |^2+\left(l\gamma+\frac{k_0 l}{m_1}\right)|\hat{\theta}_{1t}|| i\xi \hat{\omega } - l \hat{\varphi}|.
\end{align*}
Applying Young inequality, we obtain (\ref{eq10}).
\end{proof}

\begin{lem}\label{lem10''}
The functional
\begin{equation*}
\T_2(\xi,t)=Re\left(i\rho_1\xi\hat{\omega}_t \overline{\left(i\xi \hat{\omega } - l \hat{\varphi}\right)} + i\frac{\rho_1}{m_1}\xi\hat{\omega}_t\overline{\hat{\theta}_{1t}}\right),
\end{equation*}
satisfies
\begin{align}\label{eq36}
\frac{d}{dt}\T_2(\xi,t) +\frac{k_0}{2}|\xi|^2|i\xi \hat{\omega } - l \hat{\varphi}|^2 \leq &\frac{\alpha_1\rho_1}{m_1}|\xi|^3|\hat{\omega}_t||\hat{\theta}_{1t}|+Re\left( ikl\xi(i\xi\hat{\varphi}-\hat{\psi}-l\hat{\omega})\overline{\left(i\xi \hat{\omega } - l \hat{\varphi}\right)}\right) \notag \\
&+\frac{kl}{m_1}Re\left( i\xi\overline{\hat{\theta}}_{1t}\left ( i\xi \hat{\varphi} - \hat{\psi} -l \hat{\omega }  \right )\right)-\frac{k_1\rho_1}{m_1}Re\left(i\xi^3\hat{\omega}_t\overline{\hat{\theta}}_{1}\right)+C_3 |\xi|^2|\hat{\theta}_{1t}|^2,
\end{align}
where $C_3$ is a positive constant.
\end{lem}
\begin{proof}
Multiplying (\ref{eq3}) by $i\xi\overline{\left(i\xi \hat{\omega } - l \hat{\varphi}\right)}$ and taking real part, 
\begin{align*}
\frac{d}{dt} Re\left(i\rho_1\xi\hat{\omega}_{t}\overline{\left(i\xi \hat{\omega } - l \hat{\varphi}\right)}\right)&-Re\left(i\rho_1\xi\hat{\omega}_{t}\overline{\left(i\xi \hat{\omega } - l \hat{\varphi}\right)}_t\right)+k_0|\xi|^2 | i\xi \hat{\omega } - l \hat{\varphi}|^2 \notag \\
&- Re\left(ikl\xi\left ( i\xi \hat{\varphi} - \hat{\psi} -l \hat{\omega }  \right )\overline{\left(i\xi \hat{\omega } - l \hat{\varphi}\right)}\right)- Re\left( \gamma \xi^2 \hat{\theta}_{1t}\overline{\left(i\xi \hat{\omega } - l \hat{\varphi}\right)}\right)=0 ,
\end{align*}
by using (\ref{eq4}), we have
\begin{align*}
\frac{d}{dt} Re\left(i\rho_1\xi\hat{\omega}_{t}\overline{\left(i\xi \hat{\omega } - l \hat{\varphi}\right)}\right)+\frac{\rho_1}{m_1}\frac{d}{dt}Re\left(i\xi\hat{\omega} _{t}\overline{\hat{\theta}}_{1t}\right)-\frac{\rho_1}{m_1}Re\left(i\xi\hat{\omega} _{tt}\overline{\hat{\theta}}_{1t}\right) +\frac{\rho_1k_1}{m_1}Re\left(i\xi ^3\hat{\omega} _{t}\overline{\hat{\theta}}_{1}\right)& \notag \\
+\frac{\alpha_1\rho_1}{m_1}Re\left(i\xi^3\hat{\omega} _{t}\overline{\hat{\theta}}_{1t}\right)+k_0|\xi|^2 | i\xi \hat{\omega } - l \hat{\varphi}|^2 - Re\left(ikl\xi\left ( i\xi \hat{\varphi} - \hat{\psi} -l \hat{\omega }  \right )\overline{\left(i\xi \hat{\omega } - l \hat{\varphi}\right)}\right)& \notag \\
 - Re\left( \gamma \xi^2 \hat{\theta}_{1t}\overline{\left(i\xi \hat{\omega } - l \hat{\varphi}\right)}\right)&=0. 
\end{align*}
Note that, (\ref{eq3}) implies that
\begin{align*}
\frac{d}{dt}\T_2(\xi,t) &-\frac{kl}{m_1}Re\left( i\xi\overline{\hat{\theta}}_{1t}\left ( i\xi \hat{\varphi} - \hat{\psi} -l \hat{\omega }  \right )\right)  -Re \left( ikl\xi \left ( i\xi \hat{\varphi} - \hat{\psi} -l \hat{\omega }  \right)\overline{\left(i\xi \hat{\omega } - l \hat{\varphi}\right)}\right)+k_0\xi^2|i\xi \hat{\omega } - l \hat{\varphi}|^2  \\
&+\frac{k_1\rho_1}{m_1}Re\left(i\xi^3\hat{\omega}_t\overline{\hat{\theta}}_{1}\right)+\frac{\alpha_1\rho_1}{m_1}Re\left(i\xi^3\hat{\omega}_t\overline{\hat{\theta}}_{1t}\right) \leq  \frac{\gamma\xi^2}{m_1} |\hat{\theta}_{1t} |^2+\left(\gamma+\frac{k_0 }{m_1}\right)\xi^2|\hat{\theta}_{1t}|| i\xi \hat{\omega } - l \hat{\varphi}|.
\end{align*}
Applying Young inequality, we obtain (\ref{eq36}).
\end{proof}

\begin{lem}\label{lem10}
Consider the functional 
\begin{equation*}
\J_2(\xi,t):=l\T_1(\xi,t)+\T_2(\xi,t)+\frac{\rho_1 k_1}{m_1}Re\left( \xi^2\overline{\hat{\theta}}_{1}\left(i\xi\hat{\omega}-l\hat{\varphi}\right)\right).
\end{equation*}
Then, there exist $\delta > 0$ such that 
\begin{align}\label{eq41}
\frac{d}{dt}\J_2(\xi,t)+k_0\delta| i\xi \hat{\omega } - l \hat{\varphi}|^2 \leq &\frac{\alpha_1\rho_1l}{m_1} |\xi|^2|\hat{\varphi}_t||\hat{\theta}_{1t}|+\frac{\alpha_1\rho_1}{m_1} |\xi|^3|\hat{\omega}_t||\hat{\theta}_{1t}| +C_4(1+\xi^2)|\hat{\theta}_{1t}|^2
\end{align}
where $C_3$ is a positive constant.
\end{lem}


\begin{proof}
By Lemma (\ref{lem10'}) and Lemma (\ref{lem10''}),
\begin{align*}
\frac{d}{dt} \J_2(\xi,t)+\frac{k_0}{2}(l^2+\xi^2)|i\xi \hat{\omega } - l \hat{\varphi}|^2 &\leq \frac{\rho_1l\alpha_1}{m_1}|\xi|^2|\hat{\varphi}_t||\hat{\theta}_{1t}|+\frac{\rho_1k_1}{m_1}\xi^2|\hat{\theta}_{1t}||i\xi\hat{\omega}-l\hat{\varphi}| +\frac{\alpha_1\rho_1}{m_1}|\xi|^3|\hat{\omega}_t||\hat{\theta}_{1t}| \\
&+C_4(1+ \xi^2)|\hat{\theta}_{1t}|^2 
\end{align*}
Note that there exist $\delta >0$ such that $4\delta \leq \frac{l^2+\xi^2}{1+\xi^2}$. Thus,
\begin{align}\label{eq43}
\frac{d}{dt} \J_2(\xi,t)+2k_0\delta(1+\xi^2)|i\xi \hat{\omega } - l \hat{\varphi}|^2 &\leq \frac{\rho_1l\alpha_1}{m_1}|\xi|^2|\hat{\varphi}_t||\hat{\theta}_{1t}|+\frac{\rho_1k_1}{m_1}(1+\xi^2)|\hat{\theta}_{1t}||i\xi\hat{\omega}-l\hat{\varphi}| +\frac{\alpha_1\rho_1}{m_1}|\xi|^3|\hat{\omega}_t||\hat{\theta}_{1t}| \notag \\
&+C_4(1+ \xi^2)|\hat{\theta}_{1t}|^2. 
\end{align}
Applying Young inequality, (\ref{eq41}) follows.
\end{proof}

\begin{lem}\label{lem11}
Consider the functional 
\begin{equation*}
\J_3(\xi,t)=Re\left(-\rho_2 \hat{\psi}_t\overline{\left ( i\xi \hat{\varphi} - \hat{\psi} -l \hat{\omega }  \right )}-i\frac{\rho_1b}{k}\xi\hat{\psi}\overline{\hat{\varphi}_t}\right).
\end{equation*}
If $\frac{\rho_1}{\rho_2} =\frac{k}{b}$ and $k=k_0$, then
\begin{align}\label{eq12}
\frac{d}{dt}\J_3(\xi,t) +\frac{k}{2}|i\xi \hat{\varphi} -\hat{\psi} -l \hat{\omega }|^2 \leq \rho_2 |\hat{\psi}_t|^2 +\rho_2 l Re\left( \hat{\psi}_t\overline{\hat{\omega}_t}\right) -blRe \left( i\xi\hat{\psi}\overline{\left ( i\xi \hat{\omega } - l \hat{\varphi} \right )}\right) + \frac{bl\gamma}{k} |\xi| |\hat{\psi}||\hat{\theta}_{1t}|+C_5|\xi|^2|\hat{\theta}_{2t}|^2
\end{align}
Moreover, if $\frac{\rho_1}{\rho_2} \neq \frac{k}{b}$ or $k\neq k_0$, then
\begin{align}\label{eq12'}
\frac{d}{dt}\J_3(\xi,t) +\frac{k}{2}|i\xi \hat{\varphi} -\hat{\psi} -l \hat{\omega }|^2 \leq &\rho_2 |\hat{\psi}_t|^2 +\rho_2 l Re\left( \hat{\psi}_t\overline{\hat{\omega}_t}\right) +\left(\rho_2-\frac{b\rho_1}{k}\right) Re\left( i\xi\hat{\psi}_t\overline{\hat{\varphi}_t}\right) \notag\\
&-\frac{k_0bl}{k}Re \left( i\xi\hat{\psi}\overline{\left ( i\xi \hat{\omega } - l \hat{\varphi} \right )}\right)  + \frac{bl\gamma}{k} |\xi| |\hat{\psi}||\hat{\theta}_{1t}|+C_5|\xi|^2|\hat{\theta}_{2t}|^2,
\end{align}
where $C_5$ is a positive constant.
\end{lem}
\begin{proof}
Proceeding as proof of Lemma \ref{lem4}, we obtain (\ref{eq12}) and (\ref{eq12'}). 
\end{proof}

\begin{lem}\label{lem12}
Let $0<\varepsilon_1 < \frac{\rho_2l^2}{2\rho_1}$ and consider the functional
\begin{equation*}
\J_4(\xi,t)=Re\left(\frac{\rho_2^2l^2}{\rho_1}\hat{\psi}_t\overline{\hat{\psi}}-\rho_2l\hat{\omega}_t\overline{\hat{\psi}}\right).
\end{equation*}
If $\frac{\rho_1}{\rho_2} =\frac{k}{b}$ and $k=k_0$, then
\begin{align}\label{eq13}
\frac{d\J_4(\xi,t)}{dt}+b\left(\frac{\rho_2l^2}{\rho_1}-\frac{\varepsilon_1}{2}\right) \xi^2|\hat{\psi}|^2 \leq \frac{\rho_2^2l^2}{\rho_1}|\hat{\psi}_t|^2 - \rho_2lRe\left(\overline{\hat{\psi}}_t\hat{\omega}_t\right) + blRe\left( i\xi \hat{\psi}\overline{\left ( i\xi \hat{\omega } - l \hat{\varphi} \right )}\right) + C(\varepsilon_1)\left(|\hat{\theta}_{1t}|^2+|\hat{\theta}_{2t}|^2\right)
\end{align}
Moreover, If $\frac{\rho_1}{\rho_2} \neq \frac{k}{b}$ or $k\neq k_0$, then
\begin{align}\label{eq13'}
\frac{d}{dt}\J_4(\xi,t)+b\left(\frac{\rho_2l^2}{\rho_1}-\frac{\varepsilon_1}{2}\right) \xi^2|\hat{\psi}|^2 \leq & \frac{\rho_2^2l^2}{\rho_1}|\hat{\psi}_t|^2 - \rho_2lRe\left(\overline{\hat{\psi}}_t\hat{\omega}_t\right) + \frac{\rho_2k_0l}{\rho_1}Re\left( i\xi \hat{\psi}\overline{\left ( i\xi \hat{\omega } - l \hat{\varphi} \right )}\right) \notag \\
&+ C(\varepsilon_1)\left(|\hat{\theta}_{1t}|^2+|\hat{\theta}_{2t}|^2\right),
\end{align}
where $ C(\varepsilon_1)$ is a positive constant.
\end{lem}


\begin{proof}
Proceeding as proof of Lemma \ref{lem5}, we obtain  (\ref{eq13}) and (\ref{eq13'}). 
\end{proof}

\begin{lem}\label{lem13}
Let $0<\varepsilon_1 < \dfrac{\rho_2l^2}{2\rho_1}$ and consider $\K(\xi,t)=\J_3(\xi,t)+\J_4(\xi,t)$, If $\frac{\rho_1}{\rho_2} =\frac{k}{b}$ and $k=k_0$, then 
\begin{align*}
\frac{d}{dt}\K(\xi,t)+b\left(\frac{\rho_2l^2}{\rho_1}-\varepsilon_1\right) \xi^2|\hat{\psi}|^2 + \frac{k}{2}|i\xi \hat{\varphi} - \hat{\psi} -l \hat{\omega }|^2 \leq \rho_2s_1|\hat{\psi}_t|^2 + C(\varepsilon_1)|\hat{\theta}_{1t}
|^2+  C(\varepsilon_1)(1+\xi^2)|\hat{\theta}_{2t}|^2 
\end{align*}
Moreover, if $\frac{\rho_1}{\rho_2} \neq \frac{k}{b}$  or $k\neq k_0$, then
\begin{align*}
\frac{d}{dt}\K(\xi,t)+b\left(\frac{\rho_2l^2}{\rho_1}-\varepsilon_1\right) \xi^2|\hat{\psi}|^2 + \frac{k}{2}|i\xi \hat{\varphi} - \hat{\psi} -l \hat{\omega }|^2 \leq & \rho_2s_1|\hat{\psi}_t|^2 + k_0\left(\frac{\rho_2}{\rho_1}-\frac{b}{k}\right)l Re\left(i\xi \hat{\psi}(\overline{i\xi\hat{\omega}-l\hat{\varphi}})\right) \notag\\
&+ \left(\rho_2-\frac{b\rho_1}{k}\right) Re\left(i\xi \hat{\psi}_t\overline{\hat{\varphi}}_t\right)+
C(\varepsilon_1)|\hat{\theta}_{1t}|^2 \notag \\
&+  C(\varepsilon_1)(1+\xi^2)|\hat{\theta}_{2t}|^2,
\end{align*}
where $s_1= \frac{\rho_2 l^2}{\rho_1}+1$.
\end{lem}

\begin{proof}
It follows from Lemmas $\ref{lem11}$ and $\ref{lem12}$, applying Young inequality.
\end{proof}


\begin{lem}\label{lem14}
Consider the functional 
\begin{equation*}
\h(\xi,t)=\rho_1 Re\left( \left(i\xi \hat{\varphi}- \hat{\psi} -l \hat{\omega }\right)\overline{\hat{\omega}_t}\right)+\rho_1 Re\left(\left(i\xi \hat{\omega } - l \hat{\varphi}\right)\overline{\hat{\varphi}_t}\right).
\end{equation*}
If $\frac{\rho_1}{\rho_2} =\frac{k}{b}$ and $k=k_0$, then
\begin{align}\label{eq15}
\frac{d}{dt}\h(\xi,t)+\rho_1 l |\hat{\varphi}_t|^2+\frac{\rho_1 l}{2}|\hat{\omega}_t|^2 \leq \frac{\rho_2 k}{2 b l}|\hat{\psi}_t|^2 +\frac{3kl}{2}|i\xi \hat{\varphi} -\hat{\psi} -l \hat{\omega }|^2+\frac{3k_0l}{2}|i\xi \hat{\omega } - l \hat{\varphi}|^2 +C_6(1+\xi^2)|\hat{\theta}_{1t}|^2
\end{align}
Moreover, if $\frac{\rho_1}{\rho_2} \neq\frac{k}{b}$ or $k\neq k_0$,
\begin{align}\label{eq15'}
\frac{d}{dt}\h(\xi,t)+\rho_1 l |\hat{\varphi}_t|^2+\frac{\rho_1 l}{2}|\hat{\omega}_t|^2 \leq &\frac{\rho_1}{2l}|\hat{\psi}_t|^2 +C_1(k,k_0)|i\xi \hat{\varphi} -\hat{\psi} -l \hat{\omega }|^2+C_2(k,k_0)(1+\xi^2)|i\xi \hat{\omega } - l \hat{\varphi}|^2 \notag\\
&+C_6(1+\xi^2)|\hat{\theta}_{1t}|^2,
\end{align}
where $C_1(k,k_0), C_6$ are positive constants.
\end{lem}
\begin{proof}
Proceeding as proof of Lemma \ref{lem7}, we obtain (\ref{eq15}) and (\ref{eq15'}).
\end{proof}
\begin{lem}\label{lem15}
The functional 
\[
\s(\xi,t)= \frac{\gamma}{m_1}Re\left(\xi^2\hat{\theta}_{1t}\overline{\hat{\theta}}_{1}\right)+\frac{\gamma}{m_2}Re\left(\xi^2\hat{\theta}_{2t}\overline{\hat{\theta}}_{2}\right) + \frac{\gamma}{2}\xi^4\left(\frac{\alpha_1}{m_1}|\hat{\theta}_{1t}|^2+ \frac{\alpha_2}{m_2}| \hat{\theta}_{2t}|^2  \right)+ \gamma Re\left( i\xi^3\hat{\psi}\overline{\hat{\theta}}_2+ \xi^2\overline{\hat{\theta}}_1\left(i \xi \hat{\omega} - l\hat{\varphi}\right)\right)
\]
satisfies
\begin{equation}\label{eqq1}
\frac{d}{dt}\s(\xi,t)+\frac{k_1\gamma}{m_1}\xi^4|\hat{\theta}_1|^2+\frac{k_2\gamma}{m_2}\xi^4|\hat{\theta}_2|^2 \leq \gamma\xi^2 |\hat{\theta}_{1t}||i\xi\hat{\omega}-l\hat{\varphi}|+\gamma|\xi|^3|\hat{\psi}||\hat{\theta}_{2t}|+\frac{\gamma}{m_1} \xi^2|\hat{\theta}_{1t}|^2+\frac{\gamma}{m_2}\xi^2 |\hat{\theta}_{2t}|^2
\end{equation}
\end{lem}
\begin{proof}
Multiplying (\ref{eq4}) by $\frac{\gamma}{m_1}\xi^2 \overline{\hat{\theta}}_1$ and taking real part, we obtain 
\begin{align}\label{eqq2}
\frac{d}{dt}\left\lbrace \frac{\gamma}{m_1}Re\left(\xi^2\hat{\theta}_{1t}\overline{\hat{\theta}}_{1}\right) + \frac{\alpha_1\gamma}{2m_1}\xi^4|\hat{\theta}_{1t}|^2    + \gamma Re\left( \xi^2\overline{\hat{\theta}}_1\left(i \xi \hat{\omega} - l\hat{\varphi}\right)\right)\right\rbrace
&+k_1\frac{\gamma}{m_1}\xi^4|\hat{\theta}_1|^2 \notag \\
&\leq  \frac{\gamma}{m_1}\xi^2|\hat{\theta}_{1t}|^2 +\gamma\xi^2|\hat{\theta}_{1t}||i \xi \hat{\omega} - l\hat{\varphi}|.
\end{align}
Moreover, multiplying (\ref{eq5}) by $\frac{\gamma}{m_2}\xi^2 \overline{\hat{\theta}}_2$ and taking real part, 
\begin{align}\label{eqq3}
\frac{d}{dt}\left\lbrace \frac{\gamma}{m_2}Re\left(\xi^2\hat{\theta}_{2t}\overline{\hat{\theta}}_{2}\right)  + \frac{\alpha_2\gamma}{2m_2}\xi^4| \hat{\theta}_{2t}|^2+ \gamma Re\left( i\xi^3\hat{\psi}\overline{\hat{\theta}}_2\right)\right\rbrace +k_2\frac{\gamma}{m_2}\xi^4|\hat{\theta}_2|^2 \leq  \frac{\gamma}{m_2}\xi^2|\hat{\theta}_{2t}|^2+\gamma |\xi|^3|\hat{\psi}||\hat{\theta}_{2t}|. 
\end{align}
Adding (\ref{eqq2}) and (\ref{eqq3}), we obtain (\ref{eqq1})
\end{proof}
Now, Consider the functional
\begin{equation}
\LL_2(\xi,t) = \left\lbrace\begin{tabular}{l l}
$\J_1(\xi,t)+\varepsilon_2 \xi^2 \K(\xi,t)+\xi^2\J_2(\xi,t)+\varepsilon_3\xi^2\h(\xi,t)+\s(\xi,t)$, & if $\frac{\rho_1}{\rho_2} =\frac{k}{b}$ and $k=k_0$, \\
$\frac{\xi^2}{(1+\xi^2+\xi^4)}\left( \lambda_1  \varepsilon_3 \J_1+ \frac{1}{(1+\xi^2+\xi^4)}\left(\varepsilon_3\lambda_2  \xi^2\K + \xi^2\J_2+\varepsilon_3\xi^2\h +\s \right) \right)$, & if $\frac{\rho_1}{\rho_2} \neq \frac{k}{b}$ or $k\neq k_0$.
\end{tabular}\right.
\end{equation}
where $\lambda_1, \lambda_2, \varepsilon_2, \varepsilon_3$ are positive constants to be fixed later.

\begin{prop}\label{lem16}
There exist constants $M, M' > 0$ such that  if $\frac{\rho_1}{\rho_2} =\frac{k}{b}$ and $k=k_0$, then
\begin{align}\label{eq24}
\frac{d}{dt}\LL_2(\xi,t) &+M\xi^2 \left\lbrace b \xi^2|\hat{\psi}|^2 + k|i\xi \hat{\varphi} - \hat{\psi} -l \hat{\omega }|^2 +\rho_2|\hat{\psi}_t|^2+k_0|i\xi \hat{\omega } - l \hat{\varphi}|^2 + \rho_1|\hat{\varphi}_t|^2+\rho_1|\hat{\omega}_t|^2\right. \notag \\
&\left.+\frac{k_1\gamma}{m_1}\xi^2|\hat{\theta}_{1}|^2 +\frac{k_2\gamma}{m_2}\xi^2|\hat{\theta}_{2}|^2 \right\rbrace \notag \\
&\leq C(\varepsilon_1,\varepsilon_2,\varepsilon_3)(1+\xi^2+\xi^4+\xi^6)\xi^2|\hat{\theta}_{1t}|^2 +C(\varepsilon_1,\varepsilon_2)(1+\xi^2+\xi^4)|\hat{\theta}_{2t}|^2 
\end{align}
Moreover, if $\frac{\rho_1}{\rho_2} \neq \frac{k}{b}$ or $k\neq k_0$, we obtain
\begin{align}\label{eq24'}
\frac{d}{dt}\LL_2(\xi,t) &+M'\frac{\xi^4}{(1+\xi^2+\xi^4)^2} \left\lbrace b \xi^2|\hat{\psi}|^2 + k|i\xi \hat{\varphi} - \hat{\psi} -l \hat{\omega }|^2 +\rho_2|\hat{\psi}_t|^2+k_0|i\xi \hat{\omega } - l \hat{\varphi}|^2 + \rho_1|\hat{\varphi}_t|^2+\rho_1|\hat{\omega}_t|^2\right. \notag \\
&\left.+\frac{k_1\gamma}{m_1}\xi^2|\hat{\theta}_{1}|^2 +\frac{k_2\gamma}{m_2}\xi^2|\hat{\theta}_{2}|^2 \right\rbrace  \notag \\
&\leq C(\varepsilon_1,\varepsilon_3,\lambda_1,\lambda_2)(1+\xi^2)\xi^2\left(|\hat{\theta}_1|^2 +|\hat{\theta}_2|^2 \right).
\end{align}
\end{prop}

\begin{proof}
We can prove (\ref{eq24}) and (\ref{eq24'}) following the ideas used on the proof of Proposition \ref{lem8}, thus we omit some details. First, we suppose that $\frac{\rho_1}{\rho_2} =\frac{k}{b}$ and $k=k_0$. By Lemmas $\ref{lem9}$, $\ref{lem13}$, $\ref{lem10}$ and $\ref{lem14}$, it follows that
\begin{align*}
\frac{d}{dt}\left\lbrace\LL_2(\xi,t)-\s(\xi,t)\right\rbrace &+\left(\frac{\rho_2l^2}{\rho_1}-2\varepsilon_1\right)\varepsilon_2b \xi^4|\hat{\psi}|^2 + \left(\frac{\varepsilon_2}{4}-\frac{3l}{2}\varepsilon_3\right)k\xi^2|i\xi \hat{\varphi} - \hat{\psi} -l \hat{\omega }|^2 \\
&+\left(\frac{m}{2}-s_1\varepsilon_2-\frac{k}{2bl}\varepsilon_3\right)\rho_2\xi^2|\hat{\psi}_t|^2 +\left(\delta-\frac{3l}{2}\varepsilon_3\right)k_0\xi^2|i\xi \hat{\omega } - l \hat{\varphi}|^2 \\
&+ \frac{\rho_1l}{2}\varepsilon_3 |\xi|^2|\hat{\varphi}_t|^2+\frac{\rho_1l }{4}\varepsilon_3|\xi|^2|\hat{\omega}_t|^2 \notag \\
&\leq C(\varepsilon_1,\varepsilon_2,\varepsilon_3)(1+\xi^2+\xi^4+\xi^6)\xi^2|\hat{\theta}_{1t}|^2 +C(\varepsilon_1,\varepsilon_2)(1+\xi^2+\xi^4)|\hat{\theta}_{2t}|^2. 
\end{align*}
Adding $\s(\xi,t)$ in the above inequality, applying Lemma \ref{lem15} and Young inequality, we obtain
\begin{align*}
\frac{d}{dt}\LL_2(\xi,t) &+\left(\frac{\rho_2l^2}{\rho_1}-3\varepsilon_1\right)\varepsilon_2b \xi^4|\hat{\psi}|^2 + \left(\frac{\varepsilon_2}{4}-\frac{3l}{2}\varepsilon_3\right)k\xi^2|i\xi \hat{\varphi} - \hat{\psi} -l \hat{\omega }|^2 \\
&+\left(\frac{m}{2}-s_1\varepsilon_2-\frac{k}{2bl}\varepsilon_3\right)\rho_2\xi^2|\hat{\psi}_t|^2 +\left(\delta-2l\varepsilon_3\right)k_0\xi^2|i\xi \hat{\omega } - l \hat{\varphi}|^2 \\
&+ \frac{\rho_1l}{2}\varepsilon_3 |\xi|^2|\hat{\varphi}_t|^2+\frac{\rho_1l }{4}\varepsilon_3|\xi|^2|\hat{\omega}_t|^2+\frac{k_1\gamma}{m_1}\xi^4|\hat{\theta}_{1}|^2+\frac{k_2\gamma}{m_2}\xi^4|\hat{\theta}_{2}|^2 \\
&\leq C(\varepsilon_1,\varepsilon_2,\varepsilon_3)(1+\xi^2+\xi^4+\xi^6)\xi^2|\hat{\theta}_{1t}|^2 +C(\varepsilon_1,\varepsilon_2)(1+\xi^2+\xi^4)|\hat{\theta}_{2t}|^2. 
\end{align*}
We choose our constants as follows:
\[
\varepsilon_1 < \frac{\rho_2 l^2}{3 \rho_1}, \quad
\varepsilon_2 < \frac{m_2}{2s_1} \quad \text{and} \quad
\varepsilon_3 < \min\left\lbrace \frac{\delta}{2l}, \frac{\varepsilon_2}{6l},\frac{2bl}{k}\left(\frac{m_2}{2}-s_1\varepsilon_2\right)\right\rbrace
\]
Consequently, we deduce that there exist $M>0$, such that (\ref{eq24}) holds. \\

Second, we assume that  $\frac{\rho_1}{\rho_2} \neq \frac{k}{b}$ and $k\neq k_0$. By Lemmas $\ref{lem9}$, $\ref{lem13}$, the estimate (\ref{eq43}) in Lemma $\ref{lem10}$, Lemma  $\ref{lem14}$, adding these inequalities and by using Young inequality, we obtain
\begin{align*}
\frac{d}{dt}&\left\lbrace \LL_2(\xi,t)- \frac{\xi^2\s(\xi,t)}{(1+\xi^2+\xi^4)^2}\right\rbrace+\left(\frac{\rho_2l^2}{\rho_1}-3\varepsilon_1\right)\frac{\lambda_2\varepsilon_3b\xi^6}{(1+\xi^2+\xi^4)^2}|\hat{\psi}|^2  +\frac{\rho_1 \varepsilon_3l\xi^4}{4(1+\xi^2+\xi^4)^2}|\hat{\omega}_t|^2  \\
&+\left(\frac{\lambda_1m_2}{2}-C(\varepsilon_4,\lambda_2)-\frac{\rho_1}{\rho_2}\right)\frac{\rho_2\varepsilon_3\xi^4}{(1+\xi^2+\xi^4)}|\hat{\psi}_t|^2 + \left(\frac{\lambda_2}{4}-C(k,k_0)\right)\frac{\varepsilon_3k\xi^4}{(1+\xi^2+\xi^4)^2}|i\xi \hat{\varphi} - \hat{\psi} -l \hat{\omega }|^2 \\
&+\left(\delta-C(k,k_0,\varepsilon_1,\lambda_2)\varepsilon_3\right)\frac{k_0(1+\xi^2)\xi^4}{(1+\xi^2+\xi^4)^2}|i\xi \hat{\omega } - l \hat{\varphi}|^2 + 
\left(\frac{l}{2}-\frac{\varepsilon_4\lambda_2}{\rho_1}\right)\frac{\rho_1\varepsilon_3\xi^4}{(1+\xi^2+\xi^4)^2}|\hat{\varphi}_t|^2\notag \\
&\leq C(\varepsilon_1,\varepsilon_3,\lambda_1,\lambda_2)(1+\xi^2)\xi^2\left(|\hat{\theta}_{1t}|^2 +|\hat{\theta}_{2t}|^2\right).
\end{align*}
In the last estimate, we used also the inequalities (\ref{ert}). Adding $\frac{\xi^2}{(1+\xi^2+\xi^4)^2}\s(\xi,t)$ in the above inequality, applying Lemma \ref{lem15} and Young inequality, it follows that 
\begin{align*}
\frac{d}{dt}\LL_2(\xi,t)&+\left(\frac{\rho_2l^2}{\rho_1}-4\varepsilon_1\right)\frac{\lambda_2\varepsilon_3b\xi^6}{(1+\xi^2+\xi^4)^2}|\hat{\psi}|^2 + \left(\frac{\lambda_2}{4}-C(k,k_0)\right)\frac{\varepsilon_3k\xi^4}{(1+\xi^2+\xi^4)^2}|i\xi \hat{\varphi} - \hat{\psi} -l \hat{\omega }|^2  \\
&+\left(\frac{\lambda_1m_2}{2}-C(\varepsilon_4,\lambda_2)-\frac{\rho_1}{\rho_2}\right)\frac{\rho_2\varepsilon_3\xi^4}{(1+\xi^2+\xi^4)}|\hat{\psi}_t|^2  +\frac{\rho_1l \varepsilon_3\xi^4}{4(1+\xi^2+\xi^4)^2}|\hat{\omega}_t|^2 \notag\\
&+\left(\delta-C(k,k_0,\varepsilon_1,\lambda_2)\varepsilon_3\right)\frac{k_0(1+\xi^2)\xi^4}{(1+\xi^2+\xi^4)^2}|i\xi \hat{\omega } - l \hat{\varphi}|^2 + 
\left(\frac{l}{2}-\frac{\varepsilon_4\lambda_2}{\rho_1}\right)\frac{\rho_1\varepsilon_2\xi^4}{(1+\xi^2+\xi^4)^2}|\hat{\varphi}_t|^2\notag \\
&+\frac{k_1\gamma}{m_1}\frac{\xi^6}{(1+\xi^2+\xi^4)^2}|\hat{\theta}_{1}|^2+\frac{k_2\gamma}{m_2}\frac{\xi^6}{(1+\xi^2+\xi^4)^2}|\hat{\theta}_{2}|^2 \notag\\
&\leq  C(\varepsilon_1,\varepsilon_3,\lambda_1,\lambda_2)(1+\xi^2)\xi^2\left(|\hat{\theta}_{1t}|^2 +|\hat{\theta}_{2t}|^2\right).
\end{align*}
We choose our constants as follows: 
\[
\varepsilon_1 < \frac{\rho_2 l^2}{4\rho_1}, \quad
\lambda_2 > 4C(k,k_0), \quad
\varepsilon_4 < \frac{\rho_1 l}{2\lambda_2}, \quad
\lambda_1 > \frac{2(\rho_1+C(\varepsilon_4,\lambda_2)\rho_2) }{m_2 \rho_2} \quad \text{and} \quad
\varepsilon_3 < \frac{\delta}{C(k,k_0,\varepsilon_2,\lambda_2)}.
\]
Consequently, by using (\ref{ert}), we deduce that there exist $M'>0$, such that (\ref{eq24'}) holds.
\end{proof}

\begin{thm}\label{teo1'}
For any $t \geq 0$ and $\xi \in \R$, we obtain the following decay rates,
\begin{equation}\label{eq27}
\hat{\E}(\xi,t) \leq \left\lbrace\begin{tabular}{l l}
$C \hat{\E}(\xi,0) e^{-\beta s_1(\xi)t}$, & if $\frac{\rho_1}{\rho_2} =\frac{k}{b}$ and $k=k_0$, \\
$C' \hat{\E}(\xi,0) e^{-\beta' s_2(\xi)t}$, & if $\frac{\rho_1}{\rho_2} \neq\frac{k}{b}$ or $k\neq k_0,$
\end{tabular}\right.
\end{equation}
where $C, \beta, C',\beta'$ are positive constants and 
\begin{equation*}
s_1(\xi)=\frac{\xi^4}{\left(1+\xi^2+\xi^4+\xi^6+\xi^4\right)}, \quad s_2(\xi)=\frac{\xi^4}{\left(1+\xi^2\right)\left(1+\xi^2+\xi^4\right)^2}.
\end{equation*}

\end{thm}

\begin{proof}
We prove (\ref{eq27}), by Proposition \ref{lem16} and using the same approach done in the proof of Theorem \ref{teo1}, Thus, we omit the details.
\end{proof}

\section{The main result}
In this section, we establish decay estimates of the solutions to the systems (\ref{ee1}) and (\ref{ee2}). For Bresse system \eqref{ee1}, thermoelasticity of Type II, we consider the vector solution:
\begin{equation}\label{vectorsolution1}
V_1:=\left( \rho_1^{\frac{1}{2}}\varphi_{t},\rho_2^{\frac{1}{2}}\psi_{t}, \rho_1^{\frac{1}{2}}\omega_{t},\left(\frac{\gamma}{m_1}\right)^{\frac{1}{2}}\theta_1, \left(\frac{\gamma}{m_2}\right)^{\frac{1}{2}}\theta_2,b^{\frac{1}{2}}\psi_x,k^{\frac{1}{2}}\left(\varphi_x - \psi_x -l \omega_x\right), k_0^{\frac{1}{2}}\left( \omega_x - l\varphi\right)\right)
\end{equation}
and for  Bresse system \eqref{ee2}, thermoelasticity of type III, 
\begin{multline}\label{vectorsolution2}
V_2:=\left( \rho_1^{\frac{1}{2}}\varphi_{t},\rho_2^{\frac{1}{2}}\psi_{t}, \rho_1^{\frac{1}{2}}\omega_{t},\left(\frac{\gamma}{m_1}\right)^{\frac{1}{2}}\theta_{1t}, \left(\frac{k_1\gamma}{m_1}\right)^{\frac{1}{2}}\theta_{1x}, \left(\frac{\gamma}{m_2}\right)^{\frac{1}{2}}\theta_{2t},\left(\frac{k_2\gamma}{m_2}\right)^{\frac{1}{2}}\theta_{2x},b^{\frac{1}{2}}\psi_x,k^{\frac{1}{2}}\left(\varphi_x - \psi_x -l \omega_x\right), \right. \\
\left. k_0^{\frac{1}{2}}\left( \omega_x - l\varphi\right)\right)
\end{multline}
Note that
\begin{equation}\label{eq45}
\hat{E}(\xi,t)=|\hat{V}_1(\xi,t)|^2,\quad \hat{\E}(\xi,t)=|\hat{V}_2(\xi,t)|^2,
\end{equation}
where $\hat{E}$ and $\hat{\E}$ are defined in \eqref{energytypeII} and \eqref{energytypeIII}, respectively. \\

We are now in position to prove our main result. 

\begin{proof}[\textbf{Proof of Theorem \ref{teo2}}]
Applying the Plancherel identity and (\ref{eq45}), we have
\begin{align*}
\|\partial^k_xV_1(t)\|_{L^2(\R)}^2 = \|(i\xi)^k\hat{V_1}(t)\|_{L^2(\R)}^2 = \int_{\R}|\xi|^{2k}\hat{E}(\xi,t)^2 d\xi,
\end{align*}
and
\begin{align*}
\|\partial^k_xV_2(t)\|_{L^2(\R)}^2 = \|(i\xi)^k\hat{V_2}(t)\|_{L^2(\R)}^2 = \int_{\R}|\xi|^{2k}\hat{\E}(\xi,t)^2 d\xi.
\end{align*}
By Theorems \ref{teo1} and \ref{teo1'}, it yields that  
\begin{align*}
\|\partial^k_xV_j(t)\|_2^2 &\leq C\int_{\R}|\xi|^{2k}e^{-\beta s(\xi)t}\hat{V_j}(0,\xi)^2 d\xi \\
&\leq C\int_{|\xi|\leq 1}|\xi|^{2k}e^{-\beta s_i(\xi)t}\hat{V_j}^2(0,\xi) d\xi + C\int_{|\xi|\geq 1}|\xi|^{2k}e^{-\beta  s_i(\xi)t}\hat{V_j}^2(0,\xi) d\xi\\
&=I_1+I_2, \qquad \qquad (i,j=1,2).
\end{align*}
It is not difficult to see that if $\frac{\rho_1}{\rho_2} =\frac{k}{b}$ and $k=k_0$, the function $s_1(\xi)$  satisfies
\begin{equation}\label{e46}
\left\lbrace\begin{tabular}{l c l}
$s_1(\xi) \geq \frac{1}{5}\xi^4$ & if & $|\xi|\leq 1$ \\
$s_1(\xi) \geq \frac{1}{5}\xi^{-4}$ & if & $|\xi|\geq 1$ 
\end{tabular}\right.
\end{equation} 
Thus, we estimate $I_1$ as follows, 
\begin{equation}
I_1 \leq C \|\hat{V_j^0}\|_{L^\infty}^2 \int_{|\xi|\leq 1}|\xi|^{2k}e^{-\frac{\beta}{5}\xi^4t} d\xi \leq C_1\|\hat{V_j^0}\|_{L^\infty}^2\left( 1+t\right)^{-\frac{1}{4}(1+2k)} \leq C_1\left( 1+t\right)^{-\frac{1}{4}(1+2k)} \|V_j^0\|_{L^1}^2, \quad j=1,2.
\end{equation}
On the other hand, by using the second inequality in (\ref{e46}), we obtain
\begin{align*}
I_2 &\leq C\int_{|\xi|\geq 1}|\xi|^{2k}e^{-\frac{\beta}{5}\xi^{-4}t}\hat{V_j^0}^2(\xi)d\xi \leq C \sup_{|\xi| \geq 1}\{ |\xi |^{-2l}e^{-\frac{\beta}{5} \xi ^{-4}t} \} \int_{\R}|\xi|^{2(k+l)}\hat{V_j^0}^2(\xi)d\xi  \notag \\
&\leq C_2 (1+t)^{-\frac{l}{2}} \|\partial_x^{k+l}V_j^0\|_2^2,\quad j=1,2.
\end{align*}
Combining the estimates of $I_1$ and $I_2$, we obtain $(\ref{e32})$. 
On the other hand, if $\frac{\rho_1}{\rho_2} \neq \frac{k}{b}$ or $k\neq k_0$, the function $s_2(\xi)$ satisfies
\begin{equation}\label{e46'}
\left\lbrace\begin{tabular}{l c l}
$s_2(\xi) \geq \frac{1}{18}\xi^4$ & if & $|\xi|\leq 1$ \\
$s_2(\xi) \geq \frac{1}{18}\xi^{-6}$ & if & $|\xi|\geq 1$ 
\end{tabular}\right.
\end{equation}
Thus, we estimate $I_1$ as following,
\begin{equation*}
I_1 \leq C \|\hat{V_j^0}_0\|_{L^\infty}^2 \int_{|\xi|\leq 1}|\xi|^{2k}e^{-\frac{\beta}{18}\xi^4t} d\xi \leq C_1\|\hat{V_j^0}\|_{L^\infty}^2\left( 1+t\right)^{-\frac{1}{4}(1+2k)} \leq C_1\left( 1+t\right)^{-\frac{1}{4}(1+2k)}\|V_j^0\|_{L^1}^2
\end{equation*}
Moreover, by using the second inequality in (\ref{e46'}), it follows that
\begin{align*}
I_2 &\leq C\int_{|\xi|\geq 1}|\xi|^{2k}e^{-\frac{\beta}{18}\xi^{-6}t}\hat{V_j^0}^2(\xi)d\xi \leq C \sup_{|\xi| \geq 1}\{ |\xi |^{-2l}e^{-\frac{\beta}{18} \xi ^{-6}t} \} \int_{\R}|\xi|^{2(k+l)}\hat{V_j^0}^2(\xi)d\xi  \notag \\
&\leq C_2 (1+t)^{-\frac{l}{3}} \|\partial_x^{k+l}V_j^0\|_2^2
\end{align*}
Combining the estimates of $I_1$ and $I_2$, we obtain $(\ref{e32'})$.
\end{proof}

\end{document}